\documentclass{article}
\usepackage{amsthm,amsopn,amsmath}
\theoremstyle{plain}\newtheorem{theorem}{Theorem}[section]
\newtheorem{lemma}[theorem]{Lemma}

\newtheorem{corollary}[theorem]{Corollary}
\theoremstyle{definition}\newtheorem{definition}[theorem]{Definition}
\newtheorem{conjecture}[theorem]{Conjecture}
\newtheorem{example}[theorem]{Example}
\newtheorem{procedure}[theorem]{Procedure}
\theoremstyle{remark}\newtheorem*{remark}{Remark}

\newtheorem*{notabene}{Nota Bene}

\usepackage{hyperref}
\usepackage{graphics}
\usepackage{graphicx}\usepackage{epic, curves}\usepackage[usenames]{color}\usepackage{epsfig}\usepackage{curves}\usepackage{rotating}\usepackage{multirow}\usepackage{mathdots}
\title{A Parking Function Bijection supporting the Haglund-Morse-Zabrocki Conjectures}\author{Angela Hicks\\(University of California-San Diego)}
\begin{document}\maketitle
\begin{abstract} The shuffle conjecture expresses a relationship between parking functions, diagonal harmonics, and the Bergeron-Garsia $\nabla$ operator. Recent conjectures about a family of modified Hall-Littlewood operators made by Haglund, Morse, and Zabrocki sharpen the shuffle conjecture and suggest a variety of combinatorial properties of parking functions. 
In particular, their conjectures combined with previously established commutativity laws of the Hall-Littlewood operators,
suggest the existence of certain bijections relating parking functions with different diagonal compositions. In this paper we formulate a conjecture which
yields an algorithm for the construction of these bijections,
prove a special case, and give some applications.
\end{abstract} 
\section{Introduction}
We begin with a few brief definitions so that we may precisely state our result. Further background for the unfamiliar reader will be included in the following section.
A \textbf{parking function} is a
two line array $$PF=\left[\begin{matrix}r_1 &r_2&\cdots& r_n\\g_1&g_2 &\cdots &g_n\end{matrix}\right]$$ with the following properties:
\begin{itemize}
\item $g_1,g_2,\ldots ,g_n$ are nonnegative integers,
$g_1=0$ and for $i<n$, $g_{i+1}\leq g_i+1$.
\item $\left[\begin{matrix}r_1 &r_2&\cdots& r_n\end{matrix}\right]\in S_n$,
with $r_{i+1}>r_i$ when $g_{i+1}= g_i+1$.
\end{itemize}
In this paper, we refer to the columns of $PF$ (i.e.$\left[\begin{matrix}r_i\\g_i\end{matrix}\right]$.) as the \textbf{dominoes} of $PF$. 
For such a parking function, using $\chi$ for the truth function we define the following statistics:
\begin{align*}\operatorname{area}(PF)=&\sum g_i\\
\operatorname{dinv}(PF)=&\sum_{i<j}\chi (g_i=g_j\text{ and }r_i<r_j)\\&\hspace{1.5cm}+\chi(g_i+1=g_j\text{ and }r_i> r_j).\end{align*}
The \textbf{word} of a parking function ($\operatorname{word}(PF)$) is the permutation $[r_{\sigma_1},\cdots,r_{\sigma_n}]$ with the following properties:
\begin{itemize}
\item $g_{\sigma_i}\geq g_{\sigma_{i+1}}$.
\item If $g_{\sigma_i}=g_{\sigma_{i+1}}$, then $\sigma_i>\sigma_{i+1}$.
\end{itemize}
Then $$\operatorname{ides}(PF)=\{r:r+1 \text{ occurs before }r \text{ in } \operatorname{word}(PF)\}.$$ 
The \textbf{diagonal word} of a parking function ($\operatorname{diagword}(PF)$) is the permutation
$[r_{\sigma_1},\cdots,r_{\sigma_n}]$ 
with the following properties:
\begin{itemize}
\item $g_{\sigma_i}\geq g_{\sigma_{i+1}}$.
\item If $g_{\sigma_i}=g_{\sigma_{i+1}}$, then $r_{\sigma_i}<r_{\sigma_{i+1}}$.
\end{itemize}
Two parking functions $PF$ and $PF'$ have the same diagonal word if and only if the dominoes of $PF$ can be rearranged to give $PF'$. 
Finally, if $[1,z_2, \cdots, z_k]$ gives the list of $z$ such that $g_{z_i}=0$, given in increasing order,
then we define the composition of $PF$ as the vector: 
$$
\operatorname{comp}(PF)=[z_2- 1,z_3-z_2, \cdots, z_k-z_{k-1}, n-z_k+1].
$$
We will use $\mathcal{A}_{c}$ for the set of parking functions with composition $c$
and $\mathcal{A}_{c}(\tau)$ for the set of parking functions with composition $c$
and diagonal word $\tau$.
Let 
$$
c=[c_1,\cdots,c_k]=[c',c_i,c_{i+1},c''],
$$ 
where $c'$ ($c''$) denote the sequences giving the first $i-1$ (respectively last $k-i-1$) parts of $c$.
From the conjectures of Haglund-Morse-Zabrocki in \cite{HMZConj} 
we can derive the following
\begin{conjecture}\label{conj:commute1}
For $c_i\leq c_{i+1}-1$ there exists a bijection $f$
$$
f:\mathcal{A}_{[c',c_i,c_{i+1},c'']}\cup \mathcal{A}_{[c',c_{i+1}-1,c_{i}+1,c'']}
\hskip .1in \longleftrightarrow \hskip .1in 
\mathcal{A}_{[c',c_{i+1},c_i,c'']}\cup \mathcal{A}_{[c',c_{i}+1,c_{i+1}-1,c''] }
$$
with the following properties:
\begin{itemize}
\item $f$ increases the dinv by exactly one.
\item $f$ respects the ides and the area.
\end{itemize}
\end{conjecture}
Computer data and theoretical considerations led us to formulate 
the following stronger conjecture:
\begin{conjecture}\label{conj:commute2}
For $b\le a-1$ there exists a bijection $f$ 
$$
f\, :\,
\mathcal{A}_{[ b,a]}\cup \mathcal{A}_{[ a-1,b+1] }
\hskip .1in \longleftrightarrow \hskip .1in 
\mathcal{A}_{[ a,b]}\cup \mathcal{A}_{[ b+1,a-1]}
$$
with the following properties:
\begin{itemize}
\item $f$ increases the dinv by exactly one.
\item $f$ respects the ides and the diagonal word.
\item $f(\mathcal{A}_{[ a-1,b+1] })\subset \mathcal{A}_{[ a,b]} $ and $f^{-1}(\mathcal{A}_{[ b+1,a-1] })\subset \mathcal{A}_{[ b,a]} $.
\end{itemize}
\end{conjecture}
For a given permutation $\tau\in S_n$ set
$$
\mathcal{C}_c^\tau=\sum_{PF\in \mathcal{A}_{c}(\tau) } t^{\text{area}(PF)} q^{\text{dinv}(PF)}Q_{\text{ides}(PF)},
$$ 
the previous conjecture gives the obvious conjectured corollary:
\begin{conjecture}\label{conj:easyconj} For $b\leq a-1 $
and for all $\tau$ we have
$$
\mathcal{C}_{a, b}^\tau+\mathcal{C}_{b+1, a-1}^\tau=q(
\mathcal{C}^\tau_{b, a}+\mathcal{C}^\tau_{a-1, b+1})$$
\end{conjecture}
In this paper we construct a bijection that proves the special case $b=1$ of Conjecture \ref{conj:commute2} and show how the polynomial identity in \ref{conj:commute1} for $b=1$
can be used to simplify some recent parking function results.
In the following section, we examine parking functions in greater detail and explain why the conjectures of Jim Haglund, Jennifer Morse, and Mike Zabrocki in \cite{HMZConj} imply Conjecture \ref{conj:commute1} then demonstrate that Conjecture \ref{conj:commute2} is stronger than Conjecture \ref{conj:commute1}, as claimed above.
\section{Background}
\subsection{The Parking Functions}
A visual way to represent a parking functions is to start with 
an $n\times n$ grid with a ``main diagonal" running from its southwest corner to its northeast corner. Then a Dyck path is a series of north and east steps beginning in the southwest and ending in the northeast, 
never crossing the main diagonal. 
Figure \ref{fig:dandpf}(a) shows a typical example.
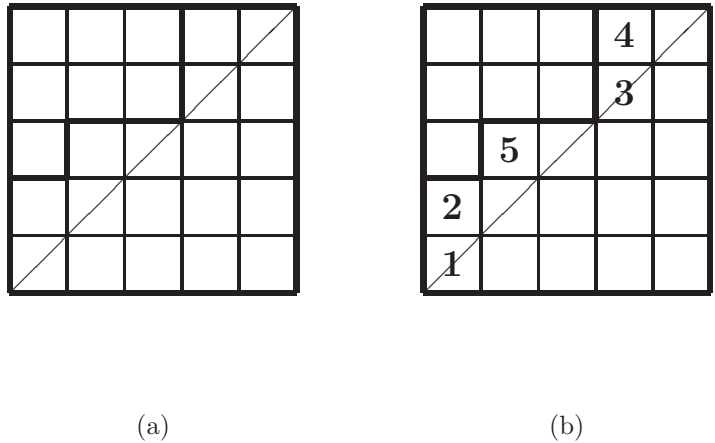
\begin{figure}\begin{center}
\begin{tabular}{p{2in}p{2in}}
\begin{center}
\setlength{\unitlength}{.15in}\begin{picture}(10,10)\linethickness{1pt}\multiput(0, 0)(2, 0){6}{\line(0, 1){10}}\multiput(0, 0)(0, 2){6}{\line(1,0){10}}\linethickness{.1pt}\put(0,0){\line(1, 1){10}}\linethickness{2pt}\multiput(0, 0)(10, 0){2}{\line(0, 1){10}}\multiput(0, 0)(0,10){2}{\line(1,0){10}}\put(0,0){\line(0, 1){2}}\put(0,2){\line(0, 1){2}}\put(2,4){\line(0, 1){2}}\put(6,6){\line(0, 1){2}}\put(6,8){\line(0, 1){2}}\put(0,2){\line(1,0){0}}\put(0,4){\line(1,0){2}}\put(2,6){\line(1,0){4}}\put(6,8){\line(1,0){0}}\end{picture}
\end{center} &\begin{center}
\setlength{\unitlength}{.15in}\begin{picture}(10,10)\linethickness{1pt}\multiput(0, 0)(2, 0){6}{\line(0, 1){10}}\multiput(0, 0)(0, 2){6}{\line(1,0){10}}\linethickness{.1pt}\put(0,0){\line(1, 1){10}}\linethickness{2pt}\multiput(0, 0)(10, 0){2}{\line(0, 1){10}}\multiput(0, 0)(0,10){2}{\line(1,0){10}}\put(0,0){\line(0, 1){2}}\put(0,2){\line(0, 1){2}}\put(2,4){\line(0, 1){2}}\put(6,6){\line(0, 1){2}}\put(6,8){\line(0, 1){2}}\put(0,2){\line(1,0){0}}\put(0,4){\line(1,0){2}}\put(2,6){\line(1,0){4}}\put(6,8){\line(1,0){0}}\Large{\put(1,1){\makebox(0,0){$\mathbf{{1}}$}}}\Large{\put(1,3){\makebox(0,0){$\mathbf{{2}}$}}}\Large{\put(3,5){\makebox(0,0){$\mathbf{{5}}$}}}\Large{\put(7,7){\makebox(0,0){$\mathbf{{3}}$}}}\Large{\put(7,9){\makebox(0,0){$\mathbf{{4}}$}}}\end{picture}
\end{center}\\
\\\begin{center} (a)\end{center} & \begin{center}(b)\end{center}\\
\end{tabular}\end{center}
\caption{A Dyck path $D\prime$ and a parking function $PF\prime$}
\label{fig:dandpf}
\end{figure}
Figure \ref{fig:dandpf}(b) gives a parking function, as a labeled Dyck path. More precisely, a parking function is a Dyck path with its north steps labeled with the integers $1$ to $n$ such that integers in the same column (such as 3 and 4 in the above example) increase from bottom to top. 
From such a diagram, we can then construct an array, (like the ones in the previous section) where $r_i$ is the integer in the $i^{\text{th}}$ row (moving bottom to top) and $g_i$ gives the number of full lattice cells in the $i^{\text{th}}$ row, between the Dyck path and the main diagonal. Thus Figure \ref{fig:dandpf} (b) corresponds to the array $$\left[\begin{matrix}1 &2&5&3&4\\0&1 &1&0 &1\end{matrix}\right].$$ 
We can clearly see now that, as mentioned in the previous section, a two line array $$PF=\left[\begin{matrix}r_1 &r_2&\cdots& r_n\\g_1&g_2 &\cdots &g_n\end{matrix}\right] $$ is a parking function exactly when:
\begin{enumerate}
\item{(Dyck Path Condition.)} For all $i$, $g_i$ is a nonnegative integer with $g_1=0$ and for $i<n$, $g_{i+1}\leq g_i+1$.
\item{(Increasing Column Condition.)} $\left[\begin{matrix}r_1 &r_2&\cdots& r_n\end{matrix}\right]$ is a permutation of $1$ to $n$ such that when $g_{i+1}= g_i+1$, $r_{i+1}>r_i$.
\end{enumerate}
To maintain the traditional parking function
jargon we will hereafter refer to $r_1,r_2,\ldots, r_n$ as the ``cars''
and to $g_1,g_2,\ldots, g_n$ as their respective ``diagonals''.
For instance we could say that ``car $r_i$ is in diagonal $g_i$''. (Notice this means that cars along the main diagonal of a parking function are in diagonal $0$.)
Although it is often more convenient to formally define statistics for parking functions when representing them as a two line array, several statistics can be computed at a glance by examining parking functions as objects in the
$n\times n$ square of lattice cells. In the following we restate our previous definitions
from this viewpoint. 
\begin{definition} The \textbf{area} of a parking function is the number of lattice squares between its Dyck path and the main diagonal. (Note here that we do not add two half squares to get a whole square.)
\end{definition}
\begin{example} For $PF\prime$ in Figure \ref{fig:dandpf} (b), $\operatorname{area}( PF\prime)=3$.
\end{example}
\begin{definition} Two cars in a parking function, are \textbf{primary attacking} if they are in the same diagonal. Two cars $a$ and $b$ are \textbf{secondary attacking} if
$a$ is to the left of $b$ and $b$ is in the diagonal just below that of $a$.
\end{definition}
\begin{example}The primary attacking pairs in $PF\prime$ are $\{1,3\}$, $\{2,5\}$, $\{2,4\}$, and $\{4,5\}$. The secondary attacking pairs are $\{2,3\}$ and $\{3,5\}$.
\end{example}
\begin{definition} Two primary attacking pairs form a \textbf{primary diagonal inversion} (dinv) when the car on the left is smaller. Two secondary attacking pairs form a 
\textbf{secondary dinv} when the car on the left (called $a$ above) is larger. The \textbf{dinv} of a parking function is the total number of pairs of cars forming either primary or secondary dinv.
\end{definition}
\begin{example}
$\{1,3\}$, $\{2,5\}$, and $\{2,4\}$ (but NOT $\{4,5\}$) form primary dinv in $PF\prime$. $\{3,5\}$ forms secondary dinv. Thus $\operatorname{dinv}(PF\prime)=4$.
\end{example}
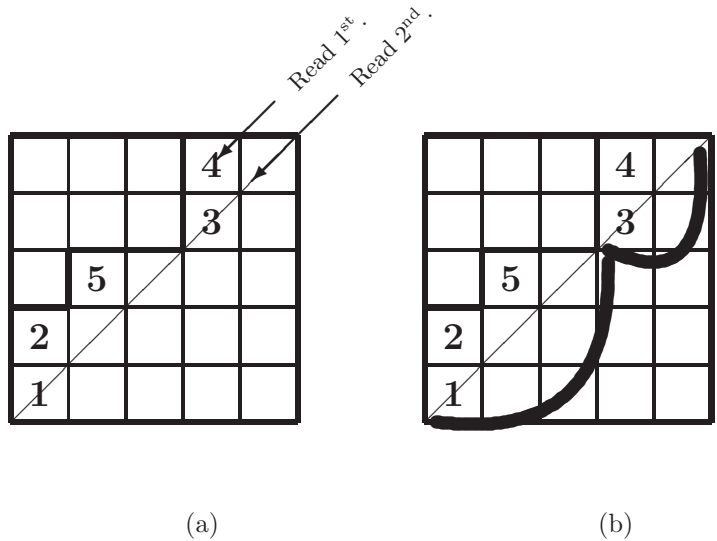
\begin{figure}\begin{center}
\begin{tabular}{p{2in}p{2in}}
\begin{center}
\setlength{\unitlength}{.15in}\begin{picture}(14,14)\linethickness{1pt}\multiput(0, 0)(2, 0){6}{\line(0, 1){10}}\multiput(0, 0)(0, 2){6}{\line(1,0){10}}\linethickness{.1pt}\put(0,0){\line(1, 1){10}}\linethickness{2pt}\multiput(0, 0)(10, 0){2}{\line(0, 1){10}}\multiput(0, 0)(0,10){2}{\line(1,0){10}}\put(0,0){\line(0, 1){2}}\put(0,2){\line(0, 1){2}}\put(2,4){\line(0, 1){2}}\put(6,6){\line(0, 1){2}}\put(6,8){\line(0, 1){2}}\put(0,2){\line(1,0){0}}\put(0,4){\line(1,0){2}}\put(2,6){\line(1,0){4}}\put(6,8){\line(1,0){0}}\Large{\put(1,1){\makebox(0,0){$\mathbf{{1}}$}}}\Large{\put(1,3){\makebox(0,0){$\mathbf{{2}}$}}}\Large{\put(3,5){\makebox(0,0){$\mathbf{{5}}$}}}\Large{\put(7,7){\makebox(0,0){$\mathbf{{3}}$}}}\Large{\put(7,9){\makebox(0,0){$\mathbf{{4}}$}}} \thicklines \put(9.2,11.2){\vector(-1, -1){2}}\small{\put(11,13){\rotatebox{45}{\makebox(0,0){Read $1^{\text{st}}$.}}}}\put(11.4,11.4){\vector(-1, -1){3}}\small{\put(13.2,13,2){\rotatebox{45}{\makebox(0,0){Read $2^{\text{nd}}$.}}}}
\end{picture}
\end{center} &
\begin{center}
\setlength{\unitlength}{.15in}\begin{picture}(14,14)\linethickness{1pt}\multiput(0, 0)(2, 0){6}{\line(0, 1){10}}\multiput(0, 0)(0, 2){6}{\line(1,0){10}}\linethickness{.1pt}\put(0,0){\line(1, 1){10}}\linethickness{2pt}\multiput(0, 0)(10, 0){2}{\line(0, 1){10}}\multiput(0, 0)(0,10){2}{\line(1,0){10}}\put(0,0){\line(0, 1){2}}\put(0,2){\line(0, 1){2}}\put(2,4){\line(0, 1){2}}\put(6,6){\line(0, 1){2}}\put(6,8){\line(0, 1){2}}\put(0,2){\line(1,0){0}}\put(0,4){\line(1,0){2}}\put(2,6){\line(1,0){4}}\put(6,8){\line(1,0){0}}\linethickness{5pt}
\curve(0.4,0, 5,1, 6.4,5.6)\curve(6.4,6, 9,6, 9.6,9.4)
\Large{\put(1,1){\makebox(0,0){$\mathbf{{1}}$}}}\Large{\put(1,3){\makebox(0,0){$\mathbf{{2}}$}}}\Large{\put(3,5){\makebox(0,0){$\mathbf{{5}}$}}}\Large{\put(7,7){\makebox(0,0){$\mathbf{{3}}$}}}\Large{\put(7,9){\makebox(0,0){$\mathbf{{4}}$}}}\end{picture}
\end{center}
\\\begin{center} (a)\end{center} & \begin{center}(b)\end{center}\\
\end{tabular}\end{center}
\caption{The parking function has word $[4,5,2,3,1]$ and composition $[3,2]$.}
\label{fig:wordandcomp}
\end{figure}
\begin{definition} The \textbf{word} of a parking function is formed by reading cars by diagonals, starting with the diagonal farthest from the main diagonal, reading cars in a diagonal from northeast to southwest. The \textbf{ides} of a parking function is the set of $r$ occurring after $r+1$ in the word.
\end{definition}
For notational convenience, using $Q$ for the Gessel quasisymmetric function (indexed by subsets of $\{1,\cdots,n-1\}$), we routinely use $$ \operatorname{wt}(PF)=t^{\text{area}(PF)} q^{\text{dinv}(PF)}Q_{\text{ides}(PF)}.$$
\begin{example}As in Figure \ref{fig:wordandcomp}, reading the integers from the diagonal in $PF\prime$ containing $4$ and then from the diagonal containing $3$, following the arrow in \ref{fig:wordandcomp} (a), we obtain the word $[4,5,2,3,1]$. The ides of $PF\prime$ is \{1,3\}.
\end{example}
\begin{definition} The \textbf{diagonal word} of a parking function ($\operatorname{diagword}(PF)$) is found by reading the diagonals, again starting with the diagonal farthest from the main diagonal, but this time recording the cars within a diagonal in increasing order.
\end{definition}
\begin{example}
The diagonal word of $PF\prime$ is $[2,4,5,1,3]$.
\end{example}
Notice that by splitting $\operatorname{diagword}(PF)$ at its descents, we get the contents of the
diagonals of $PF$. (i.e. Notice that 2, 4, and 5 are on the first diagonal and 1 and 3 are in the main diagonal of $PF\prime$.) 
Thus two parking functions have the same diagonal word exactly when they contain the same cars on every diagonal.
\begin{definition} The \textbf{composition} of a parking function gives the number of lattice cells between successive intersections of the Dyck path and the main diagonal. (See Figure \ref{fig:wordandcomp} (b).) \end{definition}
Frequently we say a car $r$ is in the $i^{\text{th}}$ part of $PF$ if $r$ 
occurs in the interval spanned by the $i^{\text{th}}$ part of $\operatorname{comp}(PF)$. 
\begin{example} $\operatorname{comp}(PF\prime)=[3,2]$, and cars 1,2, and 5 are in the first part, while cars 3 and 4 are in the
second part of $PF\prime$.
\end{example}
\subsection{Parking Functions and Diagonal Harmonics}
The bijection we demonstrate in this paper stems from recent conjectures about certain symmetric function operators and their relations to parking functions and diagonal harmonics.
\begin{definition}
The space $\text{DH}_n$ of \textbf{diagonal harmonics} is 
$$
\text{DH}_n=\{P(x,y) \text{ s. t. } \forall 1\leq h+k \leq n\text{, } \sum_{i=1}^n \partial_{x_i}^h\partial_{y_i}^k P(x,y)=0 \}.
$$
\end{definition}
We may look at forming subspaces of $\text{DH}_n$, say $\mathcal{H}_{r,s}(\text{DH}_n)$ with elements that are bi-homogeneous of degree $r$ in $x$ and $s$ in $y$. 
Then $$\text{DH}_n= {\bigoplus_s}{\bigoplus_r}_{0\leq r+s\leq \tiny{n\choose 2} }\mathcal{H}_{r,s}(\text{DH}_n).$$
\begin{definition} Using ``F char'' for the Frobenius characteristic,
Let $$\text{DH}_n[x;q,t]={\sum \sum}_{0\leq r+s\leq \tiny{ { n \choose 2 }}} t^r q^s \text{ F char }\mathcal{H}_{r,s}(\text{DH}_n).$$
\end{definition}
Almost a decade ago, Haglund et al. formulated the "shuffle" conjecture in \cite{HHLRU}, which states a relationship between the parking functions and the diagonal harmonics:
\begin{conjecture}
\begin{equation*}\text{DH}_n[x;q,t]=\sum_{PF\in \text{PF}_n} t^{\text{area}(PF)} q^{\text{dinv}(PF)}Q_{\text{ides}(PF)}.\end{equation*}
\end{conjecture}
The \textbf{nabla operator} introduced by Bergeron and Garsia in \cite{nabla}, is given by the following:
\begin{definition}
Using $\tilde{H}_\mu[X;q,t]$ for the Macdonald polynomial basis element indexed by $\mu=[\mu_1,\cdots,\mu_{n}]$ and $\mu'=[\mu_1',\cdots,\mu_{n'}']$ for the conjugate of $\mu$ 
$$\nabla \tilde{H}_\mu[X;q,t]=t^{\sum (i-1)\mu_i} q^{\sum (i-1)\mu_i'} \tilde{H}_\mu[X;q,t].$$
\end{definition}
Combining results in a number of papers (\cite{Remarkable}, \cite{Haiman}, and \cite{nabla}) it was shown that:
\begin{theorem} With $e_n$ denoting the elementary symmetric function we have $$\text{DH}_n[x;q,t]=\nabla e_n.$$
\end{theorem}
To state several more recent results in this area, we need to introduce a certain family of modified Hall Littlewood Operators. Using the brackets $[\,\,]$ to indicate plethystic substitution, set 
for any symmetric function $F[X]$
$$
C_a F[X]=\left(\frac{-1}{q}\right)^{a-1}\sum_{k\geq 0} F\left[X+\frac{1-q}{q}z\right] \bigg| _{z^k}h_{a+k}[X].
$$
The $C$ operators have several useful properties. (See \cite{Cop}.) Among them:
\begin{enumerate}
\item
Using $c\models n$ to indicate that $c$ is a composition of $n$, we have:
$$e_n=\sum_{[c_1,\cdots,c_s]\models n}C_{c_1}C_{c_2}\dots C_{c_s}1$$
\item
The $C$ operators obey the following commutativity law: 
For $a+1\leq b$, $$q(C_{a}C_b +C_{b-1}C_{a+1})=C_b C_a +C_{a+1}C_{b-1} .$$ 
\item
Using $\mu\vdash n$ to indicate that $\mu$ is a partition of $n$,
the collection $\big\{ C_{\mu_1}\cdots C_{\mu_n}1\big\}_{\mu\vdash n}$ is a basis for the homogeneous symmetric functions of degree $n$.
\end{enumerate}
Note that to simplify notation, for a composition $c=[c_1,\cdots,c_k]$, we use the convention $$C_c F[X] =C_{c_1}\dots C_{c_k} F[X].$$
Recalling that we set 
$$
c=[c_1,\cdots,c_k]=[c',c_i,c_{i+1},c''],
$$ it is then easily seen that the above commutativity relations imply that for $c_i< c_{i+1}-1$, we must
have
$$
q(C_{[c',c_i,c_{i+1},c'']}1+C_{[c',c_{i+1}-1,c_i+1,c"]}1)
=C_{[c',c_{i+1},c_i,c'']}1+C_{[c',c_i+1,c_{i+1}-1,c"]}1
$$
Now in \cite{HMZConj} Haglund, Morse and Zabrocki made the 
following conjecture:
\begin{conjecture}\label{conj:Cop} For $c$ a composition,
$$\nabla C_c 1=\sum_{PF\in \mathcal{A}_c} t^{\text{area}(PF)} q^{\text{dinv}(PF)}Q_{\text{ides}(PF)}.$$
\end{conjecture}
This conjecture in particular, when combined with the above commutativity law, implies that the parking function collections ${\mathcal A}_c$
must satisfy Conjecture \ref{conj:commute1} in the introduction.
From the definition of dinv, ides and area, it is not difficult to derive that a bijection $f$ as in Conjecture \ref{conj:commute1} with the desired properties can always be 
constructed, if, already in the two part composition case, it were possible
to construct a bijection which had the additional property
of respecting also the diagonal word. (See the following remark for further explanation.) This circumstance, together with 
strongly supporting computer data led us to our 
Conjecture \ref{conj:commute2}, namely that for $b< a-1$ there exists a bijection $f$ 
$$
f :
\mathcal{A}_{[ b,a]}\cup \mathcal{A}_{[ a-1,b+1] }
\hskip .1in \longleftrightarrow \hskip .1in 
\mathcal{A}_{[ a,b]}\cup \mathcal{A}_{[ b+1,a-1]}
$$
with the following properties:
\begin{itemize}
\item $f$ increases the dinv by exactly one.
\item $f$ respects the ides and the diagonal word.
\item $f(\mathcal{A}_{[ a-1,b+1] })\subset \mathcal{A}_{[ a,b]} $ and $f^{-1}(\mathcal{A}_{[ b+1,a-1] })\subset \mathcal{A}_{[ b,a]} $.
\end{itemize}
\begin{remark} Notice that any algorithm
which, in the two part composition case, constructs a bijections that satisfies
the dinv and ides requirements, can be imbedded in the general 
case by suitably applying it locally in the portion of the parking functions
in the $i^{th}$ and $i +1^{st}$ parts. If the resulting bijection does not 
change the diagonal of any given car, then it will not only also preserve area
but will also globally have the additional required dinv and ides properties. 
The reason for this is due to the simple fact that 
switching cars without changing their diagonal will not affect diagonal inversions nor ides formed with any car not in the parts in which the switching occurs. (See Figure \ref{fig:samediag}.)
\end{remark}
\vskip .2in
Thus we see that {Conjecture \ref{conj:commute2}} implies
{Conjecture \ref{conj:commute1}} and in particular,
our main result here which proves {Conjecture \ref{conj:commute2}}
for $b=1$ also proves {Conjecture \ref{conj:commute1}}
in the case that $c_i=1$.
\begin{figure}\label{fig:samediag}\begin{center}
\begin{tabular}{p{2in}p{2in}}
\begin{center}\setlength{\unitlength}{.1in}\begin{picture}(14,14)\linethickness{1pt}\multiput(0, 0)(2, 0){8}{\line(0, 1){14}}\multiput(0, 0)(0, 2){8}{\line(1,0){14}}\linethickness{.1pt}\put(0,0){\line(1, 1){14}}\linethickness{2pt}\multiput(0, 0)(14, 0){2}{\line(0, 1){14}}\multiput(0, 0)(0,14){2}{\line(1,0){14}}\put(0,0){\line(0, 1){2}}\put(0,2){\line(0, 1){2}}\put(4,4){\line(0, 1){2}}\put(4,6){\line(0, 1){2}}\put(4,8){\line(0, 1){2}}\put(10,10){\line(0, 1){2}}\put(10,12){\line(0, 1){2}}\put(0,2){\line(1,0){0}}\put(0,4){\line(1,0){4}}\put(4,6){\line(1,0){0}}\put(4,8){\line(1,0){0}}\put(4,10){\line(1,0){6}}\put(10,12){\line(1,0){0}}\Large{\put(1,1){\makebox(0,0){$\mathbf{{2}}$}}}\Large{\put(1,3){\makebox(0,0){$\mathbf{{6}}$}}}\Large{\put(5,5){\makebox(0,0){$\mathbf{{3}}$}}}\Large{\put(5,7){\makebox(0,0){$\mathbf{{5}}$}}}\Large{\put(5,9){\makebox(0,0){$\mathbf{{7}}$}}}\Large{\put(11,11){\makebox(0,0){$\mathbf{{1}}$}}}\Large{\put(11,13){\makebox(0,0){$\mathbf{{4}}$}}} \thicklines \put(7,7){\vector(1,1){2.5}} \thicklines \put(8,8){\vector(-1,-1){2}}\put(5,5){\circle{2.2}}\put(11,11){\circle{2.2}}\linethickness{5pt}
\curve(2.1,1, 4,1.3, 5.4,3.3)\end{picture}
\end{center} &
\begin{center}
\setlength{\unitlength}{.1in}\begin{picture}(14,14)\linethickness{1pt}\multiput(0, 0)(2, 0){8}{\line(0, 1){14}}\multiput(0, 0)(0, 2){8}{\line(1,0){14}}\linethickness{.1pt}\put(0,0){\line(1, 1){14}}\linethickness{2pt}\multiput(0, 0)(14, 0){2}{\line(0, 1){14}}\multiput(0, 0)(0,14){2}{\line(1,0){14}}\put(0,0){\line(0, 1){2}}\put(0,2){\line(0, 1){2}}\put(4,4){\line(0, 1){2}}\put(4,6){\line(0, 1){2}}\put(4,8){\line(0, 1){2}}\put(10,10){\line(0, 1){2}}\put(10,12){\line(0, 1){2}}\put(0,2){\line(1,0){0}}\put(0,4){\line(1,0){4}}\put(4,6){\line(1,0){0}}\put(4,8){\line(1,0){0}}\put(4,10){\line(1,0){6}}\put(10,12){\line(1,0){0}}\Large{\put(1,1){\makebox(0,0){$\mathbf{{2}}$}}}\Large{\put(1,3){\makebox(0,0){$\mathbf{{6}}$}}}\Large{\put(5,5){\makebox(0,0){$\mathbf{{1}}$}}}\Large{\put(5,7){\makebox(0,0){$\mathbf{{5}}$}}}\Large{\put(5,9){\makebox(0,0){$\mathbf{{7}}$}}}\Large{\put(11,11){\makebox(0,0){$\mathbf{{3}}$}}}\Large{\put(11,13){\makebox(0,0){$\mathbf{{4}}$}}}\linethickness{5pt}
\curve(2.1,1, 9,4.3, 11.4,9.6)\end{picture}
\end{center}
\\\begin{center} (a)\end{center} & \begin{center}(b)\end{center}\\
\end{tabular}\end{center}
\caption{Notice that we create a dinv between $1$ and $3$ in (b) that did not exist in (a) when we interchange two elements in the second and third part, because $1$ and $3$ remain in their original diagonal, diagonal inversions including elements in the first part of (a) (i.e. the one part not changed by our switch) still remain in (b). Thus, for example, in both (a) and (b) there is a diagonal inversion between $2$ and $3$.}
\label{fig:diagworddinv}
\end{figure}
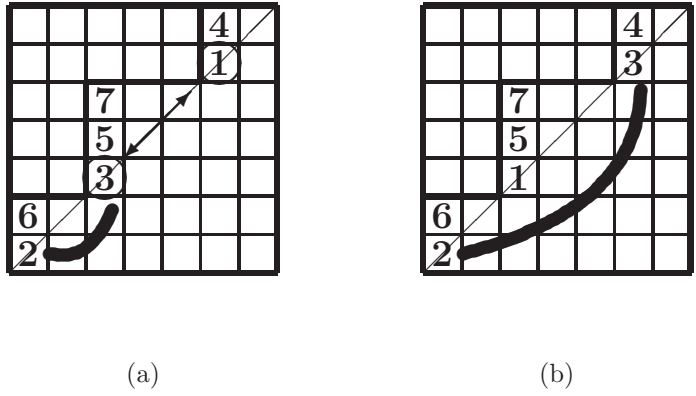
In the following sections, we will define $f$ on two part partitions, as in Conjecture \ref{conj:commute2} with $b=1$ by defining the following maps $f_1$, $f_2$, and $f_3$ and combining them to give the required bijection, $f$:
$$\begin{array}{|r||c|c|c|c|c|c|}\hline f &&\text{domain}&&\text{range}&\text{preimage} &\text{image}\\\hline\hline
&f_1&\mathcal{A}^\tau_{2, n-2}&\hookrightarrow &\mathcal{A}^\tau_{1, n-1}&\mathcal{A}^\tau_{2, n-2}&S\\ \cline{2-7}
&f_2&\mathcal{A}^\tau_{n-1, 1} &\mapsto&\mathcal{A}^\tau_{n-2, 2}&T&\mathcal{A}^\tau_{n-2, 2}
\\ \cline{2-7}
&f_3&\mathcal{A}^\tau_{n-1, 1}&\rightarrow&\mathcal{A}^\tau_{1, n-1}&\mathcal{A}^\tau_{n-1, 1}\backslash T&\mathcal{A}^\tau_{1, n-1}\backslash S
\\
\hline
\end{array}$$
In particular,  hereafter our bijection $f$ will be acting 
on
$$
PF=\left[\begin{matrix}r_1 &r_2&\cdots& r_n\\g_1&g_2 &\cdots &g_n\end{matrix}\right]= \left[\begin{matrix}D_1,\cdots, D_n\end{matrix}\right],
$$
where the $D_i$ are the previously mentioned dominoes of $PF$. We will use $\tau$ hereafter for the diagonal word of $PF$. Since $f(PF)$, as it is defined below, will always be a permutation of $D_1,\cdots, D_n$ (and thus have the same diagonal word), we will repeatedly use the following lemma.
\begin{lemma} If $PF$ and $PF'$ have the same diagonal word:
\begin{enumerate}
\item $\operatorname{area}({PF})=\operatorname{area}({PF'})$
\item $\operatorname{ides}({PF})=\operatorname{ides}({PF'})$ if and only if dominoes of the form $\left[\begin{matrix}r\\g\end{matrix} \right]$ and $\left[\begin{matrix}r+1\\g\end{matrix} \right]$ occur in same relative order in $PF$ and $PF'$.
\end{enumerate}\label{lem:easyprop}
\end{lemma}
\begin{proof} The first statement is immediate from the definition of area. Since we read the word of a parking function by diagonal, if $r$ and $r+1$ occur in distinct diagonals in $PF$ (and thus in the same distinct diagonals as $PF'$), we know they occur in the same relative order in the words of $PF$ and $PF'$.
\end{proof}
\section{The Easy Maps}
We begin by defining $f_1(PF)$ and $f_2(PF)$ as simple permutations on the dominoes of $PF$. 
\subsection{Defining $f_1$.}
Assume $PF$ has composition $[2,n-2]$. Intuitively, we currently have cars $r_1$ and $r_3$ in the main diagonal, separated by a single car, and we would like to obtain a parking function with $r_1$ and $r_3$ together.
Define the following map: 
$$f_1(PF)=\begin{cases}[D_1,D_3,D_2,D_4\cdots,D_n] & r_2>r_3\\
[D_3,D_1,D_2,D_4\cdots,D_n] & r_2<r_3\end{cases}
$$
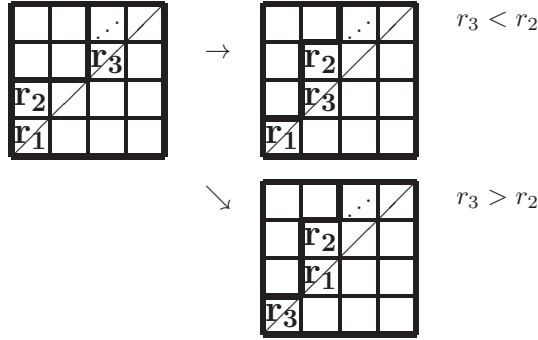
\begin{figure}
\begin{center}
\begin{tabular}{cccc}
\multirow{3}{*}{
\setlength{\unitlength}{.1in}\begin{picture}(8,8)\linethickness{1pt}\multiput(0, 0)(2, 0){5}{\line(0, 1){8}}\multiput(0, 0)(0, 2){5}{\line(1,0){8}}\linethickness{.1pt}\put(0,0){\line(1, 1){8}}\linethickness{2pt}\multiput(0, 0)(8, 0){2}{\line(0, 1){8}}\multiput(0, 0)(0,8){2}{\line(1,0){8}}\put(0,0){\line(0, 1){2}}\put(0,2){\line(0, 1){2}}\put(4,4){\line(0, 1){2}}\put(4,6){\line(0, 1){2}}\put(0,2){\line(1,0){0}}\put(0,4){\line(1,0){4}}\put(4,6){\line(1,0){0}}\Large{\put(1,1){\makebox(0,0){$\mathbf{{r_1}}$}}}\Large{\put(1,3){\makebox(0,0){$\mathbf{{r_2}}$}}}\Large{\put(5,5){\makebox(0,0){$\mathbf{{r_3}}$}}}\small{\put(5,7){\makebox(0,0){$\mathbf{{\iddots}}$}}}\end{picture}
}&&
\multirow{3}{*}{\setlength{\unitlength}{.1in}\begin{picture}(8,8)\linethickness{1pt}\multiput(0, 0)(2, 0){5}{\line(0, 1){8}}\multiput(0, 0)(0, 2){5}{\line(1,0){8}}\linethickness{.1pt}\put(0,0){\line(1, 1){8}}\linethickness{2pt}\multiput(0, 0)(8, 0){2}{\line(0, 1){8}}\multiput(0, 0)(0,8){2}{\line(1,0){8}}\put(0,0){\line(0, 1){2}}\put(2,2){\line(0, 1){2}}\put(2,4){\line(0, 1){2}}\put(4,6){\line(0, 1){2}}\put(0,2){\line(1,0){2}}\put(2,4){\line(1,0){0}}\put(2,6){\line(1,0){2}}\Large{\put(1,1){\makebox(0,0){$\mathbf{{r_1}}$}}}\Large{\put(3,3){\makebox(0,0){$\mathbf{{r_3}}$}}}\Large{\put(3,5){\makebox(0,0){$\mathbf{{r_2}}$}}}\small{\put(5,7){\makebox(0,0){$\mathbf{{\iddots}}$}}}\end{picture}
}&$r_3<r_2$\\&$\rightarrow$&\vspace{.6in}&\\\multirow{3}{*}{
\setlength{\unitlength}{.1in}}&$\searrow$&
\multirow{3}{*}{\setlength{\unitlength}{.1in}\begin{picture}(8,8)\linethickness{1pt}\multiput(0, 0)(2, 0){5}{\line(0, 1){8}}\multiput(0, 0)(0, 2){5}{\line(1,0){8}}\linethickness{.1pt}\put(0,0){\line(1, 1){8}}\linethickness{2pt}\multiput(0, 0)(8, 0){2}{\line(0, 1){8}}\multiput(0, 0)(0,8){2}{\line(1,0){8}}\put(0,0){\line(0, 1){2}}\put(2,2){\line(0, 1){2}}\put(2,4){\line(0, 1){2}}\put(4,6){\line(0, 1){2}}\put(0,2){\line(1,0){2}}\put(2,4){\line(1,0){0}}\put(2,6){\line(1,0){2}}\Large{\put(1,1){\makebox(0,0){$\mathbf{{r_3}}$}}}\Large{\put(3,3){\makebox(0,0){$\mathbf{{r_1}}$}}}\Large{\put(3,5){\makebox(0,0){$\mathbf{{r_2}}$}}}\small{\put(5,7){\makebox(0,0){$\mathbf{{\iddots}}$}}}\end{picture}
}&$r_3>r_2$\vspace{.5in}\end{tabular}
\end{center}
\caption{The $f_1$ map}\label{fig:f1}
\end{figure}
See Figure \ref{fig:f1}.
\begin{example}$f_1\left(\left[\begin{matrix}1&3&2&4&5\\0&1&0&1&1\end{matrix}\right]
\right)=\left[\begin{matrix}1&2&3&4&5\\0&0&1&1&1\end{matrix}\right]
$
\end{example}
\begin{example}$f_1\left(\left[\begin{matrix}1&2&3&4&5\\0&1&0&1&1\end{matrix}\right]\right)=\left[\begin{matrix}3&1&2&4&5\\0&0&1&1&1\end{matrix}\right]$
\end{example}
\begin{lemma} $f_1(PF)$ is in $\mathcal{A}^\tau_{1, n-1}$ . 
\end{lemma}
\begin{proof} Since $[g_1,g_2,g_3,g_4]=[0,1,0,1]$, it is easy to see that in either case $f_1(PF)$ satisfies the Dyck path condition. Similarly, the fact that $PF$ satisfies the increasing column condition implies that $f_1(PF)$ does so as well. Thus $f_1(PF)$ is a parking function. Clearly it has the required composition and diagonal word.
\end{proof}
\begin{lemma}\label{lem:f1wt} $\operatorname{wt}(PF)=q\operatorname{wt}(f_1(PF))$
\end{lemma}
\begin{proof} We only change the order of elements in a given diagonal when $r_3>r_2$, in which case $r_3$ occurs before $r_1$ in $f_1(PF)$. Since $PF$ satisfies the increasing column condition, $r_1<r_2(<r_3)$, so $r_1$ and $r_3$ differ by more than one and thus $\operatorname{ides}(PF)=\operatorname{ides}(f_1(PF))$ in either case. When $r_3<r_2$, $r_2$ and $r_3$ together form a diagonal inversion in $PF$ but not in $f_1(PF)$. Similarly, when $r_3>r_2(>r_1)$, there is a primary diagonal inversion formed between $r_1$ and $r_3$ in $PF$ but not in $f_1(PF)$. By Lemma \ref{lem:easyprop}, this is all we need to verify.
\end{proof}
Let $$S_1=\left\{\left[\begin{matrix}r_1&r_2&r_3&r_4&r_5&\dots&r_n\\0&0&1&1&g_5&\dots &g_n\end{matrix}\right]\in \mathcal{A}^\tau_{1,n-1}: r_1<r_3\text{ \& } r_2<r_4\right\}$$
and
$$S_2=\left\{\left[\begin{matrix}r_1&r_2&r_3&r_4&r_5&\dots&r_n\\0&0&1&1&g_5&\dots &g_n\end{matrix}\right]\in \mathcal{A}^\tau_{1,n-1}: r_3<r_1\text{ \& } r_1<r_4\right\}$$
\begin{lemma} $S_1$ ($S_2$) is the image of the set of parking functions with $r_2>r_3$ ($r_2<r_3$ respectively) under $f_1$.
\end{lemma}
\begin{proof}
If we reindex the cars in Figure \ref{fig:f1}, recalling that $r_1<r_2$ and $r_3<r_4$ (using the original indices here) by the increasing column condition, we have the desired result. 
\end{proof}
Let $S=S_1\cup S_2$. Since $S_1$ and $S_2$ are disjoint, it is easy to see that $f_1$ is a bijection from $C^\tau_{2, n-2}$ to $S$. 
\subsection{Defining $f_2$}
We begin by defining $g$, the inverse of $f_2$. Thus let $\operatorname{comp}(PF)=[n-2,2]$ and say that $$g(PF)=[D_1,\cdots,D_{n-2},D_n,D_{n-1}].$$See Figure \ref{fig:f2}.
\begin{figure}
\begin{center}
\begin{tabular}{ccc}
\multirow{3}{*}{
\setlength{\unitlength}{.1in}\begin{picture}(8,8)\linethickness{1pt}\multiput(0, 0)(2, 0){5}{\line(0, 1){8}}\multiput(0, 0)(0, 2){5}{\line(1,0){8}}\linethickness{.1pt}\put(0,0){\line(1, 1){8}}\linethickness{2pt}\multiput(0, 0)(8, 0){2}{\line(0, 1){8}}\multiput(0, 0)(0,8){2}{\line(1,0){8}}\put(0,0){\line(0, 1){2}}\put(0,2){\line(0, 1){2}}\put(4,4){\line(0, 1){2}}\put(4,6){\line(0, 1){2}}\put(0,2){\line(1,0){0}}\put(0,4){\line(1,0){4}}\put(4,6){\line(1,0){0}}\Large{\put(1,1){\makebox(0,0){$\mathbf{{r}}$}}}\small{\put(1,3){\makebox(0,0){$\mathbf{{\iddots}}$}}}\Large{\put(5,5){\makebox(0,0){$\mathbf{{r\prime}}$}}}\Large{\put(5,7){\makebox(0,0){$\mathbf{{r\prime\prime}}$}}}\end{picture}
}&&
\multirow{3}{*}{\setlength{\unitlength}{.1in}\begin{picture}(8,8)\linethickness{1pt}\multiput(0, 0)(2, 0){5}{\line(0, 1){8}}\multiput(0, 0)(0, 2){5}{\line(1,0){8}}\linethickness{.1pt}\put(0,0){\line(1, 1){8}}\linethickness{2pt}\multiput(0, 0)(8, 0){2}{\line(0, 1){8}}\multiput(0, 0)(0,8){2}{\line(1,0){8}}\put(0,0){\line(0, 1){2}}\put(0,2){\line(0, 1){2}}\put(2,4){\line(0, 1){2}}\put(6,6){\line(0, 1){2}}\put(0,2){\line(1,0){0}}\put(0,4){\line(1,0){2}}\put(2,6){\line(1,0){4}}\Large{\put(1,1){\makebox(0,0){$\mathbf{{r}}$}}}\small{\put(1,3){\makebox(0,0){$\mathbf{{\iddots}}$}}}\Large{\put(3,5){\makebox(0,0){$\mathbf{{r\prime\prime}}$}}}\Large{\put(7,7){\makebox(0,0){$\mathbf{{r\prime}}$}}}\end{picture}
}\\&$\rightarrow$&\vspace{.3in}\end{tabular}
\end{center}
\caption{The $g=f_2^{-1}$ map}\label{fig:f2}
\end{figure}
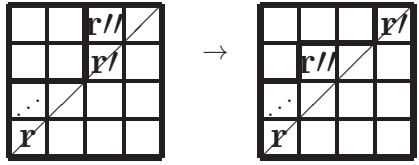
\begin{example}
$g\left(\left[\begin{matrix}1&2&3&4&5\\0&1&1&0&1\end{matrix}\right]\right)=\left[\begin{matrix}1&2&3&5&4\\0&1&1&1&0\end{matrix}\right]$
\end{example}
\begin{lemma} $g(PF)\in \mathcal{A}^\tau_{n-1,1}$ and $\operatorname{wt}(g(PF))=q\operatorname{wt}(PF)$.\end{lemma} 
\begin{proof} Since we assume $n>3$, moving $D_n$ to the left does not form any new columns, and $g(PF)$ satisfies the Dyck path and increasing column conditions. Moreover, the last two elements of $PF$, call them $r\prime$ and $r\prime\prime$ form a column so $r\prime<r\prime\prime$ and thus $r\prime$ and $r\prime\prime$ form a diagonal inversion in $g(PF)$, increasing the dinv by exactly one. Then by Lemma \ref{lem:easyprop}, $\operatorname{wt}(g(PF))=q\operatorname{wt}(PF)$.
\end{proof}
$$T=\left\{\left[\begin{matrix}r_1&r_2&r_3&\dots&r_{n-2}&r_{n-1}&r_n\\g_1&g_2&g_3&\dots&g_{n-2}& 1&0\end{matrix}\right]\in \mathcal{A}^{\tau}_{n-1,1}: r_{n-1}>r_n\right\}$$
\begin{lemma} $T$ is the image of $\mathcal{A}_{n-2,2}$ under $g$.
\end{lemma}
\begin{proof}Again, this is straightforward when we recall the increasing column condition.
\end{proof}
Now, we may define $f_2:T\mapsto \mathcal{A}^\tau_{n-2, 2}$ to be the inverse of $g$. 
\begin{example}
$f_2\left(\left[\begin{matrix}1&2&3&5&4\\0&1&1&1&0\end{matrix}\right]\right)=\left[\begin{matrix}1&2&3&4&5\\0&1&1&0&1\end{matrix}\right]$
\end{example}
\section{The remaining map}
Finally, we would like to define $f_3$, which when given a parking function of composition $[n-1,1]$ returns a parking function with composition $[1,n-1]$. A naive approach is to move the last element in the main diagonal of our parking function to the front, in a manner analogous to our construction of $f_1$. While by construction, this would give a parking function with the correct composition and ides, the result would rarely have the required singe decrease in dinv. See Figure \ref{fig:f3bad}, where such a map reduces the dinv of the parking function by two rather than the required one.
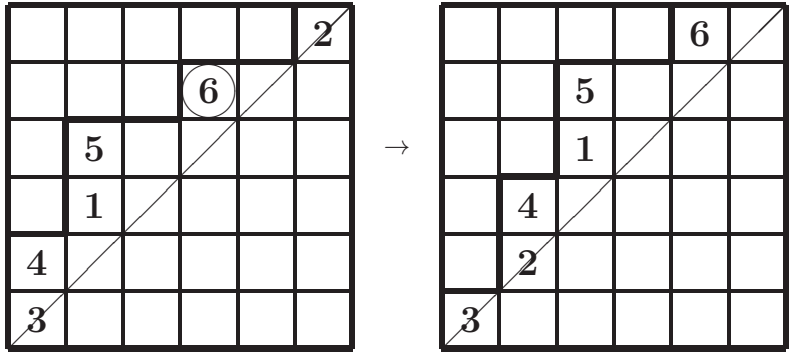
\begin{figure}
\begin{center}\vspace{.1in}
\begin{tabular}{rcl}\multirow{2}{*}{\setlength{\unitlength}{.15in}\begin{picture}(12,12)\linethickness{1pt}\multiput(0, 0)(2, 0){7}{\line(0, 1){12}}\multiput(0, 0)(0, 2){7}{\line(1,0){12}}\linethickness{.1pt}\put(0,0){\line(1, 1){12}}\linethickness{2pt}\multiput(0, 0)(12, 0){2}{\line(0, 1){12}}\multiput(0, 0)(0,12){2}{\line(1,0){12}}\put(0,0){\line(0, 1){2}}\put(0,2){\line(0, 1){2}}\put(2,4){\line(0, 1){2}}\put(2,6){\line(0, 1){2}}\put(6,8){\line(0, 1){2}}\put(10,10){\line(0, 1){2}}\put(0,2){\line(1,0){0}}\put(0,4){\line(1,0){2}}\put(2,6){\line(1,0){0}}\put(2,8){\line(1,0){4}}\put(6,10){\line(1,0){4}}\Large{\put(1,1){\makebox(0,0){$\mathbf{{3}}$}}}\Large{\put(1,3){\makebox(0,0){$\mathbf{{4}}$}}}\Large{\put(3,5){\makebox(0,0){$\mathbf{{1}}$}}}\Large{\put(3,7){\makebox(0,0){$\mathbf{{5}}$}}}\Large{\put(7,9){\makebox(0,0){$\mathbf{{6}}$}}}\Large{\put(11,11){\makebox(0,0){$\mathbf{{2}}$}}}\put(7,9){\circle{2}}\end{picture}}&&\multirow{2}{*}{\setlength{\unitlength}{.15in}\begin{picture}(12,12)\linethickness{1pt}\multiput(0, 0)(2, 0){7}{\line(0, 1){12}}\multiput(0, 0)(0, 2){7}{\line(1,0){12}}\linethickness{.1pt}\put(0,0){\line(1, 1){12}}\linethickness{2pt}\multiput(0, 0)(12, 0){2}{\line(0, 1){12}}\multiput(0, 0)(0,12){2}{\line(1,0){12}}\put(0,0){\line(0, 1){2}}\put(2,2){\line(0, 1){2}}\put(2,4){\line(0, 1){2}}\put(4,6){\line(0, 1){2}}\put(4,8){\line(0, 1){2}}\put(8,10){\line(0, 1){2}}\put(0,2){\line(1,0){2}}\put(2,4){\line(1,0){0}}\put(2,6){\line(1,0){2}}\put(4,8){\line(1,0){0}}\put(4,10){\line(1,0){4}}\Large{\put(1,1){\makebox(0,0){$\mathbf{{3}}$}}}\Large{\put(3,3){\makebox(0,0){$\mathbf{{2}}$}}}\Large{\put(3,5){\makebox(0,0){$\mathbf{{4}}$}}}\Large{\put(5,7){\makebox(0,0){$\mathbf{{1}}$}}}\Large{\put(5,9){\makebox(0,0){$\mathbf{{5}}$}}}\Large{\put(9,11){\makebox(0,0){$\mathbf{{6}}$}}}\end{picture}}\\&&\\&&\\&&\\&$\rightarrow $&\end{tabular}\end{center}\vspace{1in}\caption{A naive approach to $f_3$. The circled car is in the ``troublesome set.'' Notice the does \textit{not} include car 4, since car 4 is in the first column of the parking function.}\label{fig:f3bad}
\end{figure}
Since this unwanted change in diagonal inversions comes from elements in the first diagonal which are larger than the last car in the main diagonal, we are motivated to define the following set.
\begin{definition} Let $r_{d(i)}$ be the last car in the main diagonal of $PF$. Define the \textbf{troublesome set} of $PF$ as:
$$T(PF)=\{r_j>r_{d(i)}:2<j<d(i)\text{ and }g_j=1\}.$$
Thus the elements of $T(PF)$ are exactly those cars that form a secondary diagonal inversion with $r_{d(i)}$ and are not in the first column.\end{definition}
Speaking informally, to define $f_3(PF)$, we will need to give a series of parking functions with the same weight and diagonal word as $PF$ and progressively smaller troublesome sets. When we finally reach a parking function with an empty troublesome set, we will then be able to move the last element in the main diagonal to the front (as we do in $f_1$ and tried to do above) while changing the dinv of our parking function by exactly one, thereby gaining the required $f_3(PF)$. To find this series of parking functions, we will first find it necessary to give a third method for defining a parking function. 
\section{Changes in Dinv}
As for our previous maps, we will construct $f_3$ in such a way that we do not change the diagonal word of our parking function.
In \cite{FirstTree}, Haglund and Loehr describe a recursive operation for forming the parking functions with a given diagonal word $\tau=[\tau_1,\dots,\tau_n]$. We reproduce it here, as we will use the procedure as a starting point in producing our bijection. For notational convenience, set $$\tau'=[\tau_0',\tau_1',\dots,\tau_{n}']=[0,\tau_n,\dots,\tau_1],$$ since the procedure produces parking functions by recursively adding cars, starting with $\tau_n$ and working forward.
Using this notation, we reproduce the procedure here, using $\tau=[4,2,5,1,3]$ (and thus $\tau'=[0,3,1,5,2,4]$) as an example.
\begin{procedure}
\label{proc:treeproc}
\
\begin{enumerate}
\item Form dominoes from $\tau'$ by the following:
\begin{enumerate}
\item Split $\tau'$ at its assents to form $v$.
\begin{itemize}
\item Ex. $v=([0],[3,1],[5,2],[4])$
\end{itemize}
\item Define $t_i$ such that $\tau_i'$ is in $v_{t_i+2}$. Form the list $D=\left(\left[\begin{matrix}\tau_0'\\t_0\end{matrix}\right]\left[\begin{matrix}\tau_1'\\t_1\end{matrix}\right]\dots\left[\begin{matrix}\tau_n'\\t_n\end{matrix}\right]\right)$.
\begin{itemize}
\item Ex. $D=\left(\left[\begin{matrix}0\\-1\end{matrix}\right]\left[\begin{matrix}3\\0\end{matrix}\right]\left[\begin{matrix}1\\0\end{matrix}\right]\left[\begin{matrix}5\\1\end{matrix}\right]\left[\begin{matrix}2\\1\end{matrix}\right]\left[\begin{matrix}4\\2\end{matrix}\right]\right),$
\end{itemize}
\end{enumerate}
\item Begin with $V_0=([D_0])$. (We will remove $D_0$ from our final parking functions. Here it is a convenient way to begin our recursion.)
\begin{itemize}
\item Ex. $V_0=\left(\left[\left[\begin{matrix}0\\-1\end{matrix}\right]\right]\right)$
\end{itemize}
\item Recursively, add $D_i=\left[\begin{matrix}\tau_i'\\t_i\end{matrix}\right]$ to an element in $V_{i-1}$ in all possible ways so that $D_i$ is directly to the right of $\left[\begin{matrix}\tau_j'\\t_j\end{matrix}\right]$ and either:
\begin{enumerate}
\item $t_i=t_j$
\item $t_i=t_j+1$ and $\tau_i'>\tau_j'$.
\end{enumerate}
Form $V_i$ by adding $D_i$ in all possible ways to all the elements in $V_{i-1}$.
\begin{itemize}
\item Ex. We may add $\left[\begin{matrix}2\\1\end{matrix}\right]$ to $\left[\begin{matrix}0&3&5&1\\-1&0&1&0\end{matrix}\right]$ and get $\left[\begin{matrix}0&3&5&1&2\\-1&0&1&0&1\end{matrix}\right]$ and $\left[\begin{matrix}0&3&5&2&1\\-1&0&1&1&0\end{matrix}\right]$. 
\end{itemize}
\item Remove $\left[\begin{matrix}0\\-1\end{matrix}\right]$ from the beginning of every element in $V_n$ to form all the parking functions with diagonal word $\tau$.
\begin{itemize} \item Ex. See Figure \ref{fig:treeex} for the final family of parking functions with diagonal word $[4,2,5,1,3]$.
\end{itemize}
\end{enumerate}
\end{procedure}
\begin{figure}
\begin{tabular}{cc}\multirow{3}{*}{
\begin{sideways}\scalebox{0.7}{
\begin{tabular}{|ccccccccc|}\hline
\multicolumn{9}{|c|}{$\left[\begin{matrix}3\\0\end{matrix}\right]$}\\
\multicolumn{4}{|c}{$\left[\begin{matrix}1&3\\0&0\end{matrix}\right]$}&&\multicolumn{4}{c|}{$\left[\begin{matrix}3&1\\0&0\end{matrix}\right]$}\\
\multicolumn{2}{|c}{$\left[\begin{matrix}1&3&5\\0&0&1\end{matrix}\right]$}&\multicolumn{2}{c}{$\left[\begin{matrix}1&5&3\\0&1&0\end{matrix}\right]$ }&&\multicolumn{2}{c}{$\left[\begin{matrix}3&1&5\\0&0&1\end{matrix}\right]$}&\multicolumn{2}{c|}{$\left[\begin{matrix}3&5&1\\0&1&0\end{matrix}\right]$ }\\
$\left[\begin{matrix}1&3&5&2\\0&0&1&1\end{matrix}\right]$ & $\left[\begin{matrix}1&2&3&5\\0&1&0&1\end{matrix}\right]$ & $\left[\begin{matrix}1&5&2&3\\0&1&1&0\end{matrix}\right]$ & $\left[\begin{matrix}1&2&5&3\\0&1&1&0\end{matrix}\right]$&&$\left[\begin{matrix}3&1&5&2\\0&0&1&1\end{matrix}\right]$ & $\left[\begin{matrix}3&1&2&5\\0&0&1&1\end{matrix}\right]$ & $\left[\begin{matrix}3&5&1&2\\0&1&0&1\end{matrix}\right]$ & $\left[\begin{matrix}3&5&2&1\\0&1&1&0\end{matrix}\right]$ 
\\\hline$\left[\begin{matrix}1&3&5&2&4\\0&0&1&1&2\end{matrix}\right]$ & $\left[\begin{matrix}1&2&4&3&5\\0&1&2&0&1\end{matrix}\right]$ & $\left[\begin{matrix}1&5&2&4&3\\0&1&1&2&0\end{matrix}\right]$ & $\left[\begin{matrix}1&2&4&5&3\\0&1&2&1&0\end{matrix}\right]$&&$\left[\begin{matrix}3&1&5&2&4\\0&0&1&1&2\end{matrix}\right]$ & $\left[\begin{matrix}3&1&2&4&5\\0&0&1&2&1\end{matrix}\right]$ & $\left[\begin{matrix}3&5&1&2&4\\0&1&0&1&2\end{matrix}\right]$ & $\left[\begin{matrix}3&5&2&4&1\\0&1&1&2&0\end{matrix}\right]$ \\\hline
\end{tabular}}
\end{sideways}}
&\multirow{2}{*}{\begin{tabular}{cc}\multicolumn{2}{c}{
\setlength{\unitlength}{.15in}\begin{picture}(6,6)\linethickness{1pt}\multiput(0, 0)(2, 0){4}{\line(0, 1){6}}\multiput(0, 0)(0, 2){4}{\line(1,0){6}}\linethickness{.1pt}\put(0,0){\line(1, 1){6}}\linethickness{2pt}\multiput(0, 0)(6, 0){2}{\line(0, 1){6}}\multiput(0, 0)(0,6){2}{\line(1,0){6}}\put(0,0){\line(0, 1){2}}\put(0,2){\line(0, 1){2}}\put(4,4){\line(0, 1){2}}\put(0,2){\line(1,0){0}}\put(0,4){\line(1,0){4}}\Large{\put(1,1){\makebox(0,0){$\mathbf{{3}}$}}}\Large{\put(1,3){\makebox(0,0){$\mathbf{{5}}$}}}\Large{\put(5,5){\makebox(0,0){$\mathbf{{1}}$}}}\end{picture}}\\\setlength{\unitlength}{.15in}\begin{picture}(8,8)\linethickness{1pt}\multiput(0, 0)(2, 0){5}{\line(0, 1){8}}\multiput(0, 0)(0, 2){5}{\line(1,0){8}}\linethickness{.1pt}\put(0,0){\line(1, 1){8}}\linethickness{2pt}\multiput(0, 0)(8, 0){2}{\line(0, 1){8}}\multiput(0, 0)(0,8){2}{\line(1,0){8}}\put(0,0){\line(0, 1){2}}\put(0,2){\line(0, 1){2}}\put(2,4){\line(0, 1){2}}\put(6,6){\line(0, 1){2}}\put(0,2){\line(1,0){0}}\put(0,4){\line(1,0){2}}\put(2,6){\line(1,0){4}}\Large{\put(1,1){\makebox(0,0){$\mathbf{{3}}$}}}\Large{\put(1,3){\makebox(0,0){$\mathbf{{5}}$}}}\Large{\put(3,5){\makebox(0,0){$\mathbf{{2}}$}}}\Large{\put(7,7){\makebox(0,0){$\mathbf{{1}}$}}}\linethickness{5pt}\curve(4.1,4.55, 6.5,4.5, 7.4,5.6)\end{picture}&\setlength{\unitlength}{.15in}\begin{picture}(8,8)\linethickness{1pt}\multiput(0, 0)(2, 0){5}{\line(0, 1){8}}\multiput(0, 0)(0, 2){5}{\line(1,0){8}}\linethickness{.1pt}\put(0,0){\line(1, 1){8}}\linethickness{2pt}\multiput(0, 0)(8, 0){2}{\line(0, 1){8}}\multiput(0, 0)(0,8){2}{\line(1,0){8}}\put(0,0){\line(0, 1){2}}\put(0,2){\line(0, 1){2}}\put(4,4){\line(0, 1){2}}\put(4,6){\line(0, 1){2}}\put(0,2){\line(1,0){0}}\put(0,4){\line(1,0){4}}\put(4,6){\line(1,0){0}}\Large{\put(1,1){\makebox(0,0){$\mathbf{{3}}$}}}\Large{\put(1,3){\makebox(0,0){$\mathbf{{5}}$}}}\Large{\put(5,5){\makebox(0,0){$\mathbf{{1}}$}}}\Large{\put(5,7){\makebox(0,0){$\mathbf{{2}}$}}}
\end{picture}\\&\end{tabular}}\\&\\&\setlength{\unitlength}{.15in}\begin{picture}(20,42)\put(-5.1,38.6){ \oval(3.5,11)}\put(10,35){\bigcircle{21}}\linethickness{.1pt}\curve(-4,44, 2,44.5, 8.4,45.3)\curve(-4,33, 1.1,29,1.1,29)\put(0.5,20){\parbox{3in}{(left) The parking functions with diagonal word $[4,2,5,1,3]$ are shown here along the right column. The previous rows give the intermediate arrays formed (with the $D_0$ removed from the beginning of each array.)}}
\put(0.5,10.4){\parbox{3in}{(top, right) The top parking function has dinv $1$. If we then add a $2$ in the first diagonal, there are two possibilities: The first resulting parking function, shown here on the right, has dinv $1$ (as did its parent), but the second (the one the left) has dinv $2$. In effect, by choosing to move the $2$ further left, we are creating a dinv between the $2$ and the $1$.}}
\end{picture}\vspace{.35in}
\end{tabular}
\caption{}\label{fig:treeex}
\end{figure}
Notice that although the actual positions to which we add $D_i$ may vary depending on the particular element in $V_{n-1}$, the \textit{number} of positions to which we add $D_i$ is constant across all the elements of $V_{n-1}$. We refer to this number as $w_i$ and can calculate it directly as
\begin{align*}w_i=\#\{\tau'_j:&\tau'_j<\tau'_i \text{ and } t_j+1=t_i \}
+ \#\{\tau'_j:\tau'_j>\tau'_i \text{ and } t_j=t_i \},\end{align*}
\begin{theorem}(Haglund \& Loehr, \cite{FirstTree})
$$\sum_{\operatorname{diagword}(PF)=\tau}t^{\operatorname{area}(PF)}q^{\operatorname{dinv}(PF)}=t^{\operatorname{maj}(\tau)}\prod_{i=1}^n [w_i]_q$$
where $[n]_q=1+q+\dots+q^{n-1}$.
\end{theorem}
\begin{proof}
The proof of the theorem comes directly from the above construction. First, notice that if $\operatorname{diagword}(PF)=\tau$, $$\operatorname{area}(PF)=\sum_{i>0} t_i=\operatorname{maj}(\tau),$$ since ascents (besides the first) in $\tau'$ correspond to descents in $\tau$. Next, consider the cars in $PF$ that are placed by Procedure \ref{proc:treeproc} in $PF$ before $\tau_i'$ and form a diagonal inversion with $\tau_i'$. These are exactly the sets \begin{align*}\{\tau_j':\tau_j'>\tau_i', t_i=t_j\text{, and }
\tau_i' \text{ to the left of } \tau_j' \text{ in }PF\}\text{ and }\\\{\tau_j':\tau_j'<\tau_i', t_i=t_j+1 \text{, and }\tau_i' \text{ to the left of } \tau_j' \text{ in }PF\}.\end{align*}
Notice that $$\left\{\left[\begin{matrix}\tau_j'\\t_j\end{matrix}\right]:\tau_j'>\tau_i', t_i=t_j\right\}\text{ and }\left\{\left[\begin{matrix}\tau_j'\\t_j\end{matrix}\right]:\tau_j'<\tau_i', t_i=t_j+1\right\}$$ are exactly the elements we may place $D_i$ beside in step (3) of Procedure \ref{proc:treeproc}. Consider adding $D_i$ in all possible ways into $\pi\in V_{i-1}$. By definition, $D_i$ can be placed in $\pi$ in $w_i$ distinct places, say to the right of dominoes $D_{k_1}, \cdots, D_{k_{w_i}}$ listed in the order they appear in $\pi$. Then the diagonal inversions formed when we place $D_i$ in $\pi$ are formed between $D_i$ and the subset of $D_{k_1}, \cdots, D_{k_{w_i}}$ occurring to the right of $D_i$. Thus placing $D_i$ directly to the right of $ D_{k_{w_i}}$ (such that we see $D_{k_1}, \cdots, D_{k_{w_i}},D_i$ occurring in this relative order in the result) will not create any new diagonal inversions in $\pi$, but each time we choose to place $D_i$ further to the right, we create a new diagonal inversion, thus giving an increase in dinv of $0,1,2,\cdots,w_i-1$ as required. (Again see Figure \ref{fig:treeex}.)
\end{proof}
As it is traditionally defined, a diagonal inversion in a parking function is a set of two cars in $PF$ with certain properties. In the following, we will assign diagonal inversions to a particular car in a manner suggested by this construction (in particular the element being added in Procedure \ref{proc:treeproc} when a diagonal inversion is created). Moreover, if we choose to assign dinv in such a manner, knowing the diagonal word of a parking function and the dinv of each element identifies a unique parking function that can be constructed using the previous procedure.
\subsection{Our notation}
Using the above construction of the parking functions with a given diagonal word, we consider a set of maps which leave the diagonal word of a parking function unchanged. 
\begin{definition} Let the \textbf{dinv set} of $r_i$ in $PF$ be $$\{j>i:g_j=g_i-1\text{ \& } r_i>r_j\}\cup \{j>i:g_j=g_i \text{ \& }r_j>r_i \}.$$ Say that the \textbf{diagonal inversions} of $r_i$ in $PF$ ($\operatorname{dinv}(r_i,PF)$) is the size of the dinv set of $r_i$. Note that the dinv set then gives the index of the dominoes to the right of $r_i$ that we could have placed $\left[\begin{matrix}r_i\\d_i\end{matrix}\right]$ directly to the left of in Procedure \ref{proc:treeproc}.
\end{definition}
\begin{definition} Let the \textbf{degree set} of $r_i$ in $PF$ be $$\{j:g_j=g_i-1\text{ \& } r_i>r_j\}\cup \{j:g_j=g_i \text{ \& }r_j>r_i \}.$$ (Include $0$ in the degree set of elements in the main diagonal.) Thus the degree set gives the index of \textit{all} dominoes that we could have placed $\left[\begin{matrix}r_i\\d_i\end{matrix}\right]$ directly to the left of in Procedure \ref{proc:treeproc}. Say that $r_i$ is of \textbf{full degree} if $\operatorname{dinv}(r_i,PF)$ is one less than the size of the degree set of $r_i$. This then corresponds to placing $r_i$ as far to the left as possible when we construct $PF$.
\end{definition}
\begin{remark}Notice that if $\tau=\operatorname{diagword}(PF)$ and $\tau_j'=r_i$, then $w_j=|\operatorname{degset}(r_i,PF)|$.\end{remark}
In the following, we will define maps within the parking functions with a certain diagonal word by ``changing the dinv" of $r_i$ in $PF$ by $j$. The above construction gives us that such a parking function exists and is unique, as long as $0\leq \operatorname{dinv}(r_i,PF)+j < |\operatorname{degset}(r_i,PF)|$.
\begin{definition}
Formally, let $\operatorname{dinvinc}(r_i,PF)$ ($\operatorname{dinvdec}(r_i,PF)$) give the unique parking function $PF'$ with the following properties:
\begin{enumerate}
\item $\operatorname{diag}(PF')=\operatorname{diag}(PF)$ 
\item $\operatorname{dinv}(r_i,PF')=\operatorname{dinv}(r_i,PF)+1$ (or $\operatorname{dinv}(r_i,PF')=\operatorname{dinv}(r_i,PF)-1$ respectively)
\item For all $j\neq i$, $\operatorname{dinv}(r_j,PF)=\operatorname{dinv}(r_j,PF')$
\end{enumerate}
\end{definition}
Using this new notation, as a first exercise we may equivalently define $f_1$ and $f_2$ as specific dinv decreases. In particular, (assuming we begin with $PF$ in the proper domain) $$f_1(PF)=\begin{cases}\operatorname{dinvdec}(r_2,PF)& r_2>r_3\\
\operatorname{dinvdec}(r_1,PF)& r_2<r_3\end{cases}
$$ and
$$f_2(PF)=\operatorname{dinvdec}(r_{n-1},PF).$$
Since frequently below we will simultaneously increase the dinv of one label and decrease the dinv of another, we shorten the notation by saying that:
$$\operatorname{dinvchange}(r_i,r_j,PF)=\operatorname{dinvinc}(r_i,\operatorname{dinvdec}(r_j,PF))$$
\begin{remark}To recall which dinv we increase and which we decrease, notice that the relative position of $r_i$ and $r_j$ reflect the direction they will be moved in our parking function by such a change, since increasing the dinv of $r_i$ will cause $r_i$ to generally move towards the left, while $r_j$ will similarly move to the right. Notice also that $\operatorname{dinvinc}$ and $\operatorname{dinvdec}$ are clearly commutative. 
\end{remark}
With these conventions, we will define $f_3$ as a recursive series of specific dinv changes (which in fact hold the total dinv of the parking function fixed) which each decrease the size of the troublesome set, followed by a final dinv decrease when we have a parking function with empty troublesome set.
\begin{notabene}
Before we begin, we should notice that a dinv change, since it is defined as a change to some (and often not the last) recursive step we use to create a parking function by means of Procedure \ref{proc:treeproc}, can cause dramatic changes in our parking functions. While for a specific choice of $PF$, $\operatorname{dinvchange}(r_i,r_j,PF)$ is a permutation of the dominoes of $PF$, the permutation depends not only on the choice of $r_i$ and $r_j$ but also depends very strongly on $PF$, unlike $f_1$ or $f_2$. \end{notabene}
\begin{example}See Figure \ref{fig:earlyex}. Notice that even if we begin with almost identical parking functions, our resulting path (and even composition) may differ fairly dramatically. \end{example}
\begin{figure}\label{fig:earlyex}
\begin{center}
\begin{tabular}{ccc}
\multirow{3}{*}{\setlength{\unitlength}{.1in}{\begin{picture}(16,16)\linethickness{1pt}\multiput(0, 0)(2, 0){9}{\line(0, 1){16}}\multiput(0, 0)(0, 2){9}{\line(1,0){16}}\linethickness{.1pt}\put(0,0){\line(1, 1){16}}\linethickness{2pt}\multiput(0, 0)(16, 0){2}{\line(0, 1){16}}\multiput(0, 0)(0,16){2}{\line(1,0){16}}\put(0,0){\line(0, 1){2}}\put(0,2){\line(0, 1){2}}\put(2,4){\line(0, 1){2}}\put(2,6){\line(0, 1){2}}\put(4,8){\line(0, 1){2}}\put(4,10){\line(0, 1){2}}\put(10,12){\line(0, 1){2}}\put(14,14){\line(0, 1){2}}\put(0,2){\line(1,0){0}}\put(0,4){\line(1,0){2}}\put(2,6){\line(1,0){0}}\put(2,8){\line(1,0){2}}\put(4,10){\line(1,0){0}}\put(4,12){\line(1,0){6}}\put(10,14){\line(1,0){4}}\Large{\put(1,1){\makebox(0,0){$\mathbf{{3}}$}}}\Large{\put(1,3){\makebox(0,0){$\mathbf{{6}}$}}}\Large{\put(3,5){\makebox(0,0){$\mathbf{{4}}$}}}\Large{\put(3,7){\makebox(0,0){$\mathbf{{5}}$}}}\Large{\put(5,9){\makebox(0,0){$\mathbf{{7}}$}}}\Large{\put(5,11){\makebox(0,0){$\mathbf{{8}}$}}}\Large{\put(11,13){\makebox(0,0){$\mathbf{{1}}$}}}\Large{\put(15,15){\makebox(0,0){$\mathbf{{2}}$}}}\put(1,3){\circle{2}}\put(3,5){\circle{2}}\end{picture}}}&&
\multirow{3}{*}{\setlength{\unitlength}{.1in}\begin{picture}(16,16)\linethickness{1pt}\multiput(0, 0)(2, 0){9}{\line(0, 1){16}}\multiput(0, 0)(0, 2){9}{\line(1,0){16}}\linethickness{.1pt}\put(0,0){\line(1, 1){16}}\linethickness{2pt}\multiput(0, 0)(16, 0){2}{\line(0, 1){16}}\multiput(0, 0)(0,16){2}{\line(1,0){16}}\put(0,0){\line(0, 1){2}}\put(0,2){\line(0, 1){2}}\put(0,4){\line(0, 1){2}}\put(6,6){\line(0, 1){2}}\put(6,8){\line(0, 1){2}}\put(6,10){\line(0, 1){2}}\put(6,12){\line(0, 1){2}}\put(12,14){\line(0, 1){2}}\put(0,2){\line(1,0){0}}\put(0,4){\line(1,0){0}}\put(0,6){\line(1,0){6}}\put(6,8){\line(1,0){0}}\put(6,10){\line(1,0){0}}\put(6,12){\line(1,0){0}}\put(6,14){\line(1,0){6}}\Large{\put(1,1){\makebox(0,0){$\mathbf{{3}}$}}}\Large{\put(1,3){\makebox(0,0){$\mathbf{{4}}$}}}\Large{\put(1,5){\makebox(0,0){$\mathbf{{5}}$}}}\Large{\put(7,7){\makebox(0,0){$\mathbf{{2}}$}}}\Large{\put(7,9){\makebox(0,0){$\mathbf{{6}}$}}}\Large{\put(7,11){\makebox(0,0){$\mathbf{{7}}$}}}\Large{\put(7,13){\makebox(0,0){$\mathbf{{8}}$}}}\Large{\put(13,15){\makebox(0,0){$\mathbf{{1}}$}}}\put(7,9){\circle{2}}\put(1,3){\circle{2}}\end{picture}}\\&&\\&&\\&$\rightarrow$&\vspace{.9in}\end{tabular}\end{center}
\begin{center}
\begin{tabular}{ccc}
\multirow{3}{*}{
\setlength{\unitlength}{.1in}\begin{picture}(16,16)\linethickness{1pt}\multiput(0, 0)(2, 0){9}{\line(0, 1){16}}\multiput(0, 0)(0, 2){9}{\line(1,0){16}}\linethickness{.1pt}\put(0,0){\line(1, 1){16}}\linethickness{2pt}\multiput(0, 0)(16, 0){2}{\line(0, 1){16}}\multiput(0, 0)(0,16){2}{\line(1,0){16}}\put(0,0){\line(0, 1){2}}\put(0,2){\line(0, 1){2}}\put(2,4){\line(0, 1){2}}\put(2,6){\line(0, 1){2}}\put(4,8){\line(0, 1){2}}\put(4,10){\line(0, 1){2}}\put(10,12){\line(0, 1){2}}\put(14,14){\line(0, 1){2}}\put(0,2){\line(1,0){0}}\put(0,4){\line(1,0){2}}\put(2,6){\line(1,0){0}}\put(2,8){\line(1,0){2}}\put(4,10){\line(1,0){0}}\put(4,12){\line(1,0){6}}\put(10,14){\line(1,0){4}}\Large{\put(1,1){\makebox(0,0){$\mathbf{{3}}$}}}\Large{\put(1,3){\makebox(0,0){$\mathbf{{6}}$}}}\Large{\put(3,5){\makebox(0,0){$\mathbf{{4}}$}}}\Large{\put(3,7){\makebox(0,0){$\mathbf{{5}}$}}}\Large{\put(5,9){\makebox(0,0){$\mathbf{{1}}$}}}\Large{\put(5,11){\makebox(0,0){$\mathbf{{8}}$}}}\Large{\put(11,13){\makebox(0,0){$\mathbf{{7}}$}}}\Large{\put(15,15){\makebox(0,0){$\mathbf{{2}}$}}}\put(3,5){\circle{2}}\put(1,3){\circle{2}}\end{picture}}&&
\multirow{3}{*}{\setlength{\unitlength}{.1in}\begin{picture}(16,16)\linethickness{1pt}\multiput(0, 0)(2, 0){9}{\line(0, 1){16}}\multiput(0, 0)(0, 2){9}{\line(1,0){16}}\linethickness{.1pt}\put(0,0){\line(1, 1){16}}\linethickness{2pt}\multiput(0, 0)(16, 0){2}{\line(0, 1){16}}\multiput(0, 0)(0,16){2}{\line(1,0){16}}\put(0,0){\line(0, 1){2}}\put(0,2){\line(0, 1){2}}\put(0,4){\line(0, 1){2}}\put(2,6){\line(0, 1){2}}\put(2,8){\line(0, 1){2}}\put(10,10){\line(0, 1){2}}\put(10,12){\line(0, 1){2}}\put(12,14){\line(0, 1){2}}\put(0,2){\line(1,0){0}}\put(0,4){\line(1,0){0}}\put(0,6){\line(1,0){2}}\put(2,8){\line(1,0){0}}\put(2,10){\line(1,0){8}}\put(10,12){\line(1,0){0}}\put(10,14){\line(1,0){2}}\Large{\put(1,1){\makebox(0,0){$\mathbf{{3}}$}}}\Large{\put(1,3){\makebox(0,0){$\mathbf{{4}}$}}}\Large{\put(1,5){\makebox(0,0){$\mathbf{{5}}$}}}\Large{\put(3,7){\makebox(0,0){$\mathbf{{1}}$}}}\Large{\put(3,9){\makebox(0,0){$\mathbf{{8}}$}}}\Large{\put(11,11){\makebox(0,0){$\mathbf{{2}}$}}}\Large{\put(11,13){\makebox(0,0){$\mathbf{{6}}$}}}\Large{\put(13,15){\makebox(0,0){$\mathbf{{7}}$}}}\put(1,3){\circle{2}}\put(11,13){\circle{2}}\end{picture}}\\&&\\&&\\&$\rightarrow$&\vspace{.9in}\end{tabular}\end{center}
\caption{For two similar $PF$ given on the left, $PF'=\operatorname{dinvchange}(4,6,PF)$ is shown on the right.}
\end{figure}
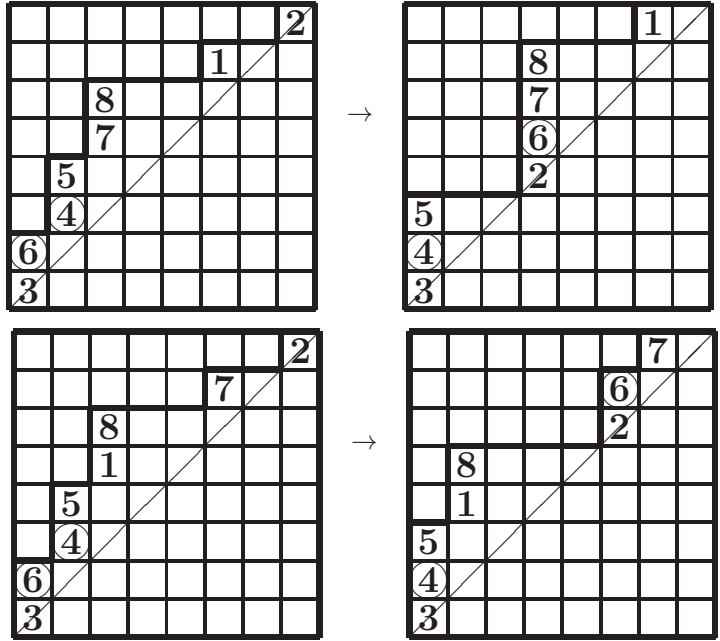
Where possible, we try to give descriptions of some of the resulting changes in our parking functions, but we suggest the reader be aware that by necessity this is a simplification of the full changes that are most easily described in full rigor using changes in the dinv. In particular, when defining $f_3$ we include a number of diagrams to give the reader some idea of how a specific step changes a parking function. For a dinv change as depicted in Figure \ref{fig:earlyex}, for example, we would use the diagram
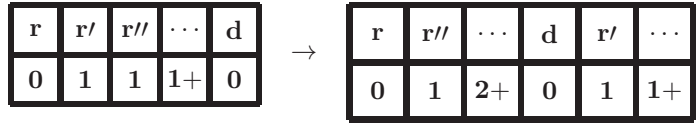
\begin{center}
\begin{tabular}{rcl}\multirow{2}{*}{\setlength{\unitlength}{.13in}\begin{picture}(10,4)\linethickness{2pt}\multiput(0, 0)(2, 0){6}{\line(0, 1){4}}\multiput(0, 0)(0, 2){3}{\line(1, 0){10}}{\put(1,3){\makebox(0,0){$\mathbf{{r}}$}}}{\put(3,3){\makebox(0,0){$\mathbf{{r\prime}}$}}}{\put(5,3){\makebox(0,0){$\mathbf{{r\prime\prime}}$}}}{\put(7,3){\makebox(0,0){$\mathbf{{\dots}}$}}}{\put(9,3){\makebox(0,0){$\mathbf{{d}}$}}}{\put(1,1){\makebox(0,0){$\mathbf{{0}}$}}}{\put(3,1){\makebox(0,0){$\mathbf{{1}}$}}}{\put(5,1){\makebox(0,0){$\mathbf{{1}}$}}}{\put(7,1){\makebox(0,0){$\mathbf{{1+}}$}}}{\put(9,1){\makebox(0,0){$\mathbf{{0}}$}}}\end{picture}}&&\multirow{2}{*}{\setlength{\unitlength}{.15in}\begin{picture}(12,4)\linethickness{2pt}\multiput(0, 0)(2, 0){7}{\line(0, 1){4}}\multiput(0, 0)(0, 2){3}{\line(1, 0){12}}{\put(1,3){\makebox(0,0){$\mathbf{{r}}$}}}{\put(3,3){\makebox(0,0){$\mathbf{{r\prime\prime}}$}}}{\put(5,3){\makebox(0,0){$\mathbf{{\dots}}$}}}{\put(7,3){\makebox(0,0){$\mathbf{{d}}$}}}{\put(9,3){\makebox(0,0){$\mathbf{{r\prime}}$}}}{\put(11,3){\makebox(0,0){$\mathbf{{\dots}}$}}}{\put(1,1){\makebox(0,0){$\mathbf{{0}}$}}}{\put(3,1){\makebox(0,0){$\mathbf{{1}}$}}}{\put(5,1){\makebox(0,0){$\mathbf{{2+}}$}}}{\put(7,1){\makebox(0,0){$\mathbf{{0}}$}}}{\put(9,1){\makebox(0,0){$\mathbf{{1}}$}}}{\put(11,1){\makebox(0,0){$\mathbf{{1+}}$}}}\end{picture}}\\&$\rightarrow $&\end{tabular}\end{center}
\vspace{.6cm}
\noindent to show specifically where $3$, $6$, $4$, and $2$ ($r$, $r\prime$, $r\prime\prime$, and $d$ respectively) move under $\operatorname{dinvchange}(4,6,PF)$. We use dominoes of the form $\left[\begin{matrix}\dots\\g+\end{matrix}\right]$and to indicate that, depending on our choice of $PF$, there may (or may not) be additional dominoes remaining in this space, and that they are on diagonal $g$ or higher.
\section{The Remaining Map Defined}
It remains to define $f_3$, a bijection between $\mathcal{C}^\tau_{n-1, 1}\backslash T$ and $\mathcal{C}^\tau_{1, n-1}\backslash S$. We construct a family of parking functions $$PF^i=\left[\begin{matrix}r_1^i &\dots&r_{n-2}^i &r_{n-1}^i &r_{n}^i\\g_1^i&\dots&g_{n-2}^i& g_{n-1}^i&g_n^i\end{matrix}\right]$$ with $PF^1=PF$ and the final result being $f_3(PF)$. Roughly speaking, we would like to move the second element on the main diagonal of $PF$ left until it is in the second spot in the resulting parking function, so that we have a parking function with composition $[1,n-1]$. Ideally, we would like the result to have an increase of exactly one dinv. 
We follow what is essentially a two step process. First, we move elements on the first diagonal and larger than $r_n$ (cars in the troublesome set) to the right of $r_n$ using a series of steps chosen because they respect the dinv of the parking function. Second, we move $r_n$ to either the first or second place in the parking function, moving it to the former only if the latter is not a parking function. It is in this last step that we lose the required one dinv.
\subsection{Notation}
We specifically identify certain elements in a parking function repeatedly in our following procedure.
For notational convenience:
\begin{itemize}
\item Let $d(i)$ be the position of the second label in the main diagonal of $PF^i$. 
\item Refer to the elements $r_j^i$ such that $r_j^i>r_{d(i)}^i$ as the ``big'' elements in $PF_i$ and refer to the elements $r_j^i$ such that $r_j^i<r_{d(i)}^i$ as the ``small'' elements in $PF_i$. 
\item Let $b(i)$ give the index of the last big label in the first part and the first diagonal.
\item Let $b'(i)$ give the index of the second to last big label in the first part and the first diagonal.
\item Let $a(i)$ give the index of the first label in the first part and the first diagonal after $b(i)$. (Note that by definition of $b(i)$, $a(i)$ will by necessity be the index of a small element.)
\item Let $m(i)$ give the index of the last label in the first part and diagonal of $PF^i$.
\item Let $p(i)$ give the index of the last label in the first part and second diagonal of $PF^i$ that is larger than $r_{d(i)+1}$ and occurs to the right of $r_{m(i)}$.
\end{itemize}
Notice that none of these previously defined cars necessarily exist for a given parking function. See Figure \ref{fig:note} for an example.
\begin{figure}\begin{center}\begin{tabular}{cc}$\left[\begin{matrix}1&8&2&9&7&3&10&6&4&5\\0&1&1&2&1&1&2&2&0&1\end{matrix}\right]$&
\begin{tabular}{|r||c|c|c|c|c|c|}
\hline
&d&b&b'&a&m&p\\
\hline
\hline
Index&9&5&2&6&6&8\\ 
\hline
Car&4&7&8&3&3&6\\
\hline
\end{tabular}\end{tabular}\end{center}
\caption{The table on the right gives the value of each letter, as defined by our new notation (say $\alpha(i)$), then the corresponding car in the parking function ($r_{\alpha(i)}$).}\label{fig:note}\end{figure}
We should make one last comment before giving the recursive procedure that defines $f_3$. Implicitly, in the following construction, we will assume that $g_{b(i)+1}^i\neq 0$ (i.e.\ that there is some car between $r_{b(i)}^i$ and $r_{d(i)}^i$) a fact which will be extremely important in proving the bijectively of $f$. We will refer to this as the \textbf{recursive condition} and will soon show that not only all objects in the domain of $f_3$ but also any intermediate parking functions formed in the following procedure will satisfy this condition.
\subsection{Constructing the image of $PF$ under $f_3$}
\begin{procedure}\label{k=1} Let $PF^1$ be a parking function in $\mathcal{A}^\tau_{n-1,1}\backslash T$ and repeat the following to produce $f_3(PF^1)$. 
\begin{enumerate}
\item If $PF^i$ has no big elements in the first part and first diagonal (besides possibly a big element in its first column), then construct
$$f_3(PF^1)=\begin{cases} \operatorname{dinvdec}(r_2^i,PF^i) & r_2^i>r_{d(i)}^i\\
\operatorname{dinvdec}(r_1^i,PF^i) & r_2^i<r_{d(i)}^i
\end{cases}.$$
When $r_2^i>r_{d(i)}^i$, \begin{tabular}{rcl}\multirow{2}{*}{\setlength{\unitlength}{.13in}\begin{picture}(10,4)\linethickness{2pt}\multiput(0, 0)(2, 0){6}{\line(0, 1){4}}\multiput(0, 0)(0, 2){3}{\line(1, 0){10}}{\put(1,3){\makebox(0,0){$\mathbf{{r}}$}}}{\put(3,3){\makebox(0,0){$\mathbf{{r'}}$}}}{\put(5,3){\makebox(0,0){$\mathbf{{\dots}}$}}}{\put(7,3){\makebox(0,0){$\mathbf{{d}}$}}}{\put(9,3){\makebox(0,0){$\mathbf{{\dots}}$}}}{\put(1,1){\makebox(0,0){$\mathbf{{0}}$}}}{\put(3,1){\makebox(0,0){$\mathbf{{1}}$}}}{\put(5,1){\makebox(0,0){$\mathbf{{1+}}$}}}{\put(7,1){\makebox(0,0){$\mathbf{{0}}$}}}{\put(9,1){\makebox(0,0){$\mathbf{{1+}}$}}}\linethickness{1pt}
\curve(7,4,4.5,4.8,2,4)\thicklines\put(2.6,4.5){\vector(-3,-2){.8}}\end{picture}}&&\multirow{2}{*}{\setlength{\unitlength}{.13in}\begin{picture}(8,4)\linethickness{2pt}\multiput(0, 0)(2, 0){5}{\line(0, 1){4}}\multiput(0, 0)(0, 2){3}{\line(1, 0){8}}{\put(1,3){\makebox(0,0){$\mathbf{{r}}$}}}{\put(3,3){\makebox(0,0){$\mathbf{{d}}$}}}{\put(5,3){\makebox(0,0){$\mathbf{{r'}}$}}}{\put(7,3){\makebox(0,0){$\mathbf{{\dots}}$}}}{\put(1,1){\makebox(0,0){$\mathbf{{0}}$}}}{\put(3,1){\makebox(0,0){$\mathbf{{0}}$}}}{\put(5,1){\makebox(0,0){$\mathbf{{1}}$}}}{\put(7,1){\makebox(0,0){$\mathbf{{1+}}$}}}\end{picture}}\\&$\rightarrow $&\end{tabular}\vspace{1cm}
\\When $r_2^i<r_{d(i)}^i$,
\begin{tabular}{rcl}\multirow{2}{*}{\setlength{\unitlength}{.13in}\begin{picture}(10,4)\linethickness{2pt}\multiput(0, 0)(2, 0){6}{\line(0, 1){4}}\multiput(0, 0)(0, 2){3}{\line(1, 0){10}}{\put(1,3){\makebox(0,0){$\mathbf{{r}}$}}}{\put(3,3){\makebox(0,0){$\mathbf{{r'}}$}}}{\put(5,3){\makebox(0,0){$\mathbf{{\dots}}$}}}{\put(7,3){\makebox(0,0){$\mathbf{{d}}$}}}{\put(9,3){\makebox(0,0){$\mathbf{{\dots}}$}}}{\put(1,1){\makebox(0,0){$\mathbf{{0}}$}}}{\put(3,1){\makebox(0,0){$\mathbf{{1}}$}}}{\put(5,1){\makebox(0,0){$\mathbf{{1+}}$}}}{\put(7,1){\makebox(0,0){$\mathbf{{0}}$}}}{\put(9,1){\makebox(0,0){$\mathbf{{1+}}$}}}\linethickness{1pt}
\curve(7,4,3.5,4.8,0,4)\thicklines\put(0.6,4.5){\vector(-3,-2){.8}}\end{picture}}&&\multirow{2}{*}{\setlength{\unitlength}{.13in}\begin{picture}(8,4)\linethickness{2pt}\multiput(0, 0)(2, 0){5}{\line(0, 1){4}}\multiput(0, 0)(0, 2){3}{\line(1, 0){8}}{\put(1,3){\makebox(0,0){$\mathbf{{d}}$}}}{\put(3,3){\makebox(0,0){$\mathbf{{r}}$}}}{\put(5,3){\makebox(0,0){$\mathbf{{r'}}$}}}{\put(7,3){\makebox(0,0){$\mathbf{{\dots}}$}}}{\put(1,1){\makebox(0,0){$\mathbf{{0}}$}}}{\put(3,1){\makebox(0,0){$\mathbf{{0}}$}}}{\put(5,1){\makebox(0,0){$\mathbf{{1}}$}}}{\put(7,1){\makebox(0,0){$\mathbf{{1+}}$}}}\end{picture}}\\&$\rightarrow $&\end{tabular}\vspace{1cm}
\item If $g_{b(i)+1}^i=1$, let $$PF^{i+1}=\operatorname{dinvchange}(r_{a(i)}^i,r_{b(i)}^i,PF^i).$$\begin{center}
\begin{tabular}{rcl}\multirow{2}{*}{\setlength{\unitlength}{.13in}\begin{picture}(14,4)\linethickness{2pt}\multiput(0, 0)(2, 0){8}{\line(0, 1){4}}\multiput(0, 0)(0, 2){3}{\line(1, 0){14}}{\put(1,3){\makebox(0,0){$\mathbf{{r}}$}}}{\put(3,3){\makebox(0,0){$\mathbf{{\dots}}$}}}{\put(5,3){\makebox(0,0){$\mathbf{{b}}$}}}{\put(7,3){\makebox(0,0){$\mathbf{{a}}$}}}{\put(9,3){\makebox(0,0){$\mathbf{{\dots}}$}}}{\put(11,3){\makebox(0,0){$\mathbf{{d}}$}}}{\put(13,3){\makebox(0,0){$\mathbf{{\dots}}$}}}{\put(1,1){\makebox(0,0){$\mathbf{{0}}$}}}{\put(3,1){\makebox(0,0){$\mathbf{{1+}}$}}}{\put(5,1){\makebox(0,0){$\mathbf{{1}}$}}}{\put(7,1){\makebox(0,0){$\mathbf{{1}}$}}}{\put(9,1){\makebox(0,0){$\mathbf{{1+}}$}}}{\put(11,1){\makebox(0,0){$\mathbf{{0}}$}}}{\put(13,1){\makebox(0,0){$\mathbf{{1+}}$}}}\linethickness{1pt}
\curve(5,4,8.5,5,12,4)\thicklines\put(11.2,4.4){\vector(3,-1){.8}}\linethickness{1pt}
\curve(7,4,5.5,4.8,4,4)\thicklines\put(4.6,4.5){\vector(-3,-2){.8}}\end{picture}}&&\multirow{2}{*}{\setlength{\unitlength}{.13in}\begin{picture}(14,4)\linethickness{2pt}\multiput(0, 0)(2, 0){8}{\line(0, 1){4}}\multiput(0, 0)(0, 2){3}{\line(1, 0){14}}{\put(1,3){\makebox(0,0){$\mathbf{{r}}$}}}{\put(3,3){\makebox(0,0){$\mathbf{{\dots}}$}}}{\put(5,3){\makebox(0,0){$\mathbf{{a}}$}}}{\put(7,3){\makebox(0,0){$\mathbf{{\dots}}$}}}{\put(9,3){\makebox(0,0){$\mathbf{{d}}$}}}{\put(11,3){\makebox(0,0){$\mathbf{{b}}$}}}{\put(13,3){\makebox(0,0){$\mathbf{{\dots}}$}}}{\put(1,1){\makebox(0,0){$\mathbf{{0}}$}}}{\put(3,1){\makebox(0,0){$\mathbf{{1+}}$}}}{\put(5,1){\makebox(0,0){$\mathbf{{1}}$}}}{\put(7,1){\makebox(0,0){$\mathbf{{2+}}$}}}{\put(9,1){\makebox(0,0){$\mathbf{{0}}$}}}{\put(11,1){\makebox(0,0){$\mathbf{{1}}$}}}{\put(13,1){\makebox(0,0){$\mathbf{{1+}}$}}}\end{picture}}\\&$\rightarrow $&\end{tabular}
\end{center}\vspace{1cm}
\item If $g_{b(i)+1}^i=2$ and $g_{b(i)-1}^i>1$, let $$PF^{i+1}=\operatorname{dinvchange}(r_{b(i)+1}^i,r_{b(i)}^i,PF^i).$$
\begin{center}\begin{tabular}{rcl}\multirow{2}{*}{\setlength{\unitlength}{.13in}\begin{picture}(16,4)\linethickness{2pt}\multiput(0, 0)(2, 0){9}{\line(0, 1){4}}\multiput(0, 0)(0, 2){3}{\line(1, 0){16}}{\put(1,3){\makebox(0,0){$\mathbf{{r}}$}}}{\put(3,3){\makebox(0,0){$\mathbf{{\dots}}$}}}{\put(5,3){\makebox(0,0){$\mathbf{{r\prime}}$}}}{\put(7,3){\makebox(0,0){$\mathbf{{b}}$}}}{\put(9,3){\makebox(0,0){$\mathbf{{r\prime\prime}}$}}}{\put(11,3){\makebox(0,0){$\mathbf{{\dots}}$}}}{\put(13,3){\makebox(0,0){$\mathbf{{d}}$}}}{\put(15,3){\makebox(0,0){$\mathbf{{\dots}}$}}}{\put(1,1){\makebox(0,0){$\mathbf{{0}}$}}}{\put(3,1){\makebox(0,0){$\mathbf{{1+}}$}}}{\put(5,1){\makebox(0,0){$\mathbf{{2}}$}}}{\put(7,1){\makebox(0,0){$\mathbf{{1}}$}}}{\put(9,1){\makebox(0,0){$\mathbf{{2}}$}}}{\put(11,1){\makebox(0,0){$\mathbf{{1+}}$}}}{\put(13,1){\makebox(0,0){$\mathbf{{0}}$}}}{\put(15,1){\makebox(0,0){$\mathbf{{1+}}$}}}\linethickness{1pt}
\curve(7,4,10.5,5,14,4)\thicklines\put(13.2,4.4){\vector(3,-1){.8}}\linethickness{1pt}
\curve(9,4,7.5,4.8,5.9,4)\thicklines\put(6.5,4.5){\vector(-3,-2){.8}}
\end{picture}}&&\multirow{2}{*}{\setlength{\unitlength}{.13in}\begin{picture}(16,4)\linethickness{2pt}\multiput(0, 0)(2, 0){9}{\line(0, 1){4}}\multiput(0, 0)(0, 2){3}{\line(1, 0){16}}{\put(1,3){\makebox(0,0){$\mathbf{{r}}$}}}{\put(3,3){\makebox(0,0){$\mathbf{{\dots}}$}}}{\put(5,3){\makebox(0,0){$\mathbf{{r\prime}}$}}}{\put(7,3){\makebox(0,0){$\mathbf{{r\prime\prime}}$}}}{\put(9,3){\makebox(0,0){$\mathbf{{\dots}}$}}}{\put(11,3){\makebox(0,0){$\mathbf{{d}}$}}}{\put(13,3){\makebox(0,0){$\mathbf{{b}}$}}}{\put(15,3){\makebox(0,0){$\mathbf{{\dots}}$}}}{\put(1,1){\makebox(0,0){$\mathbf{{0}}$}}}{\put(3,1){\makebox(0,0){$\mathbf{{1+}}$}}}{\put(5,1){\makebox(0,0){$\mathbf{{2}}$}}}{\put(7,1){\makebox(0,0){$\mathbf{{2}}$}}}{\put(9,1){\makebox(0,0){$\mathbf{{1+}}$}}}{\put(11,1){\makebox(0,0){$\mathbf{{0}}$}}}{\put(13,1){\makebox(0,0){$\mathbf{{1}}$}}}{\put(15,1){\makebox(0,0){$\mathbf{{1+}}$}}}\end{picture}}\\&$\rightarrow $&\end{tabular}
\end{center}\vspace{1cm}
\item If $g_{b(i)+1}^i=2$, $g_{b(i)-1}^i=1$, and $r_{b(i)-1}^i<r_{b(i)+1}^i$, let $$PF^{i+1}=\operatorname{dinvchange}(r_{b(i)+1}^i,r_{b(i)}^i,PF^i).$$
\begin{center}\begin{tabular}{rcl}\multirow{2}{*}{\setlength{\unitlength}{.13in}\begin{picture}(16,4)\linethickness{2pt}\multiput(0, 0)(2, 0){9}{\line(0, 1){4}}\multiput(0, 0)(0, 2){3}{\line(1, 0){16}}{\put(1,3){\makebox(0,0){$\mathbf{{r}}$}}}{\put(3,3){\makebox(0,0){$\mathbf{{\dots}}$}}}{\put(5,3){\makebox(0,0){$\mathbf{{r\prime}}$}}}{\put(7,3){\makebox(0,0){$\mathbf{{b}}$}}}{\put(9,3){\makebox(0,0){$\mathbf{{r\prime\prime}}$}}}{\put(11,3){\makebox(0,0){$\mathbf{{\dots}}$}}}{\put(13,3){\makebox(0,0){$\mathbf{{d}}$}}}{\put(15,3){\makebox(0,0){$\mathbf{{\dots}}$}}}{\put(1,1){\makebox(0,0){$\mathbf{{0}}$}}}{\put(3,1){\makebox(0,0){$\mathbf{{1+}}$}}}{\put(5,1){\makebox(0,0){$\mathbf{{1}}$}}}{\put(7,1){\makebox(0,0){$\mathbf{{1}}$}}}{\put(9,1){\makebox(0,0){$\mathbf{{2}}$}}}{\put(11,1){\makebox(0,0){$\mathbf{{1+}}$}}}{\put(13,1){\makebox(0,0){$\mathbf{{0}}$}}}{\put(15,1){\makebox(0,0){$\mathbf{{1+}}$}}}\curve(7,4,10.5,5,14,4)\thicklines\put(13.2,4.4){\vector(3,-1){.8}}\linethickness{1pt}
\curve(9,4,7.5,4.8,5.9,4)\thicklines\put(6.5,4.5){\vector(-3,-2){.8}}\end{picture}}&&\multirow{2}{*}{\setlength{\unitlength}{.13in}\begin{picture}(16,4)\linethickness{2pt}\multiput(0, 0)(2, 0){9}{\line(0, 1){4}}\multiput(0, 0)(0, 2){3}{\line(1, 0){16}}{\put(1,3){\makebox(0,0){$\mathbf{{r}}$}}}{\put(3,3){\makebox(0,0){$\mathbf{{\dots}}$}}}{\put(5,3){\makebox(0,0){$\mathbf{{r\prime}}$}}}{\put(7,3){\makebox(0,0){$\mathbf{{r\prime\prime}}$}}}{\put(9,3){\makebox(0,0){$\mathbf{{\dots}}$}}}{\put(11,3){\makebox(0,0){$\mathbf{{d}}$}}}{\put(13,3){\makebox(0,0){$\mathbf{{b}}$}}}{\put(15,3){\makebox(0,0){$\mathbf{{\dots}}$}}}{\put(1,1){\makebox(0,0){$\mathbf{{0}}$}}}{\put(3,1){\makebox(0,0){$\mathbf{{1+}}$}}}{\put(5,1){\makebox(0,0){$\mathbf{{1}}$}}}{\put(7,1){\makebox(0,0){$\mathbf{{2}}$}}}{\put(9,1){\makebox(0,0){$\mathbf{{1+}}$}}}{\put(11,1){\makebox(0,0){$\mathbf{{0}}$}}}{\put(13,1){\makebox(0,0){$\mathbf{{1}}$}}}{\put(15,1){\makebox(0,0){$\mathbf{{1+}}$}}}\end{picture}}\\&$\rightarrow $&\end{tabular}
\end{center}\vspace{1cm}
\item If $g_{b(i)+1}^i=2$, $g_{b(i)-1}^i=1$, $r_{b(i)-1}^i>r_{b(i)+1}^i$, and in addition $b(i)-2\neq 1$, let $$PF^{i+1}=\operatorname{dinvchange}(r_{b(i)}^i,r_{b(i)-1}^i,PF^i).$$
\begin{center}\begin{tabular}{rcl}\multirow{2}{*}{\setlength{\unitlength}{.13in}\begin{picture}(16,4)\linethickness{2pt}\multiput(0, 0)(2, 0){9}{\line(0, 1){4}}\multiput(0, 0)(0, 2){3}{\line(1, 0){16}}{\put(1,3){\makebox(0,0){$\mathbf{{r}}$}}}{\put(3,3){\makebox(0,0){$\mathbf{{\dots}}$}}}{\put(5,3){\makebox(0,0){$\mathbf{{r\prime}}$}}}{\put(7,3){\makebox(0,0){$\mathbf{{b}}$}}}{\put(9,3){\makebox(0,0){$\mathbf{{r\prime\prime}}$}}}{\put(11,3){\makebox(0,0){$\mathbf{{\dots}}$}}}{\put(13,3){\makebox(0,0){$\mathbf{{d}}$}}}{\put(15,3){\makebox(0,0){$\mathbf{{\dots}}$}}}{\put(1,1){\makebox(0,0){$\mathbf{{0}}$}}}{\put(3,1){\makebox(0,0){$\mathbf{{1+}}$}}}{\put(5,1){\makebox(0,0){$\mathbf{{1}}$}}}{\put(7,1){\makebox(0,0){$\mathbf{{1}}$}}}{\put(9,1){\makebox(0,0){$\mathbf{{2}}$}}}{\put(11,1){\makebox(0,0){$\mathbf{{1+}}$}}}{\put(13,1){\makebox(0,0){$\mathbf{{0}}$}}}{\put(15,1){\makebox(0,0){$\mathbf{{1+}}$}}}\linethickness{1pt}
\curve(5,4,9.5,5,14,4)\thicklines\put(13.2,4.4){\vector(3,-1){.8}}\linethickness{1pt}
\curve(7,4,5.5,4.8,4,4)\thicklines\put(4.6,4.5){\vector(-3,-2){.8}}\end{picture}}&&\multirow{2}{*}{\setlength{\unitlength}{.13in}\begin{picture}(16,4)\linethickness{2pt}\multiput(0, 0)(2, 0){9}{\line(0, 1){4}}\multiput(0, 0)(0, 2){3}{\line(1, 0){16}}{\put(1,3){\makebox(0,0){$\mathbf{{r}}$}}}{\put(3,3){\makebox(0,0){$\mathbf{{\dots}}$}}}{\put(5,3){\makebox(0,0){$\mathbf{{b}}$}}}{\put(7,3){\makebox(0,0){$\mathbf{{r\prime\prime}}$}}}{\put(9,3){\makebox(0,0){$\mathbf{{\dots}}$}}}{\put(11,3){\makebox(0,0){$\mathbf{{d}}$}}}{\put(13,3){\makebox(0,0){$\mathbf{{r\prime}}$}}}{\put(15,3){\makebox(0,0){$\mathbf{{\dots}}$}}}{\put(1,1){\makebox(0,0){$\mathbf{{0}}$}}}{\put(3,1){\makebox(0,0){$\mathbf{{1+}}$}}}{\put(5,1){\makebox(0,0){$\mathbf{{1}}$}}}{\put(7,1){\makebox(0,0){$\mathbf{{2}}$}}}{\put(9,1){\makebox(0,0){$\mathbf{{2+}}$}}}{\put(11,1){\makebox(0,0){$\mathbf{{0}}$}}}{\put(13,1){\makebox(0,0){$\mathbf{{1}}$}}}{\put(15,1){\makebox(0,0){$\mathbf{{1+}}$}}}\end{picture}}\\&$\rightarrow $&\end{tabular}
\end{center}\vspace{1cm}
\item If $g_{b(i)+1}^i=2$, $g_{b(i)-1}^i=1$, $r_{b(i)-1}^i>r_{b(i)+1}^i$, and in addition $b(i)-2=1$ and $r_1^i<r_{b(i)}^i$, let $$PF^{i+1}=\operatorname{dinvchange}(r_{b(i)}^i,r_{b(i)-1}^i,PF^i).$$
\begin{center}\begin{tabular}{rcl}\multirow{2}{*}{\setlength{\unitlength}{.13in}\begin{picture}(14,4)\linethickness{2pt}\multiput(0, 0)(2, 0){8}{\line(0, 1){4}}\multiput(0, 0)(0, 2){3}{\line(1, 0){14}}{\put(1,3){\makebox(0,0){$\mathbf{{r}}$}}}{\put(3,3){\makebox(0,0){$\mathbf{{r\prime}}$}}}{\put(5,3){\makebox(0,0){$\mathbf{{b}}$}}}{\put(7,3){\makebox(0,0){$\mathbf{{r\prime\prime}}$}}}{\put(9,3){\makebox(0,0){$\mathbf{{\dots}}$}}}{\put(11,3){\makebox(0,0){$\mathbf{{d}}$}}}{\put(13,3){\makebox(0,0){$\mathbf{{\dots}}$}}}{\put(1,1){\makebox(0,0){$\mathbf{{0}}$}}}{\put(3,1){\makebox(0,0){$\mathbf{{1}}$}}}{\put(5,1){\makebox(0,0){$\mathbf{{1}}$}}}{\put(7,1){\makebox(0,0){$\mathbf{{2}}$}}}{\put(9,1){\makebox(0,0){$\mathbf{{1+}}$}}}{\put(11,1){\makebox(0,0){$\mathbf{{0}}$}}}{\put(13,1){\makebox(0,0){$\mathbf{{1+}}$}}}\linethickness{1pt}
\curve(3,4,7.5,5,12,4)\thicklines\put(11.2,4.4){\vector(3,-1){.8}}\linethickness{1pt}
\curve(5,4,3.5,4.8,2,4)\thicklines\put(2.6,4.5){\vector(-3,-2){.8}}\end{picture}}&&\multirow{2}{*}{\setlength{\unitlength}{.13in}\begin{picture}(14,4)\linethickness{2pt}\multiput(0, 0)(2, 0){8}{\line(0, 1){4}}\multiput(0, 0)(0, 2){3}{\line(1, 0){14}}{\put(1,3){\makebox(0,0){$\mathbf{{r}}$}}}{\put(3,3){\makebox(0,0){$\mathbf{{b}}$}}}{\put(5,3){\makebox(0,0){$\mathbf{{r\prime\prime}}$}}}{\put(7,3){\makebox(0,0){$\mathbf{{\dots}}$}}}{\put(9,3){\makebox(0,0){$\mathbf{{d}}$}}}{\put(11,3){\makebox(0,0){$\mathbf{{r\prime}}$}}}{\put(13,3){\makebox(0,0){$\mathbf{{\dots}}$}}}{\put(1,1){\makebox(0,0){$\mathbf{{0}}$}}}{\put(3,1){\makebox(0,0){$\mathbf{{1}}$}}}{\put(5,1){\makebox(0,0){$\mathbf{{2}}$}}}{\put(7,1){\makebox(0,0){$\mathbf{{2+}}$}}}{\put(9,1){\makebox(0,0){$\mathbf{{0}}$}}}{\put(11,1){\makebox(0,0){$\mathbf{{1}}$}}}{\put(13,1){\makebox(0,0){$\mathbf{{1+}}$}}}\end{picture}}\\&$\rightarrow $&\end{tabular}
\end{center}\vspace{1cm}
\item If $g_{b(i)+1}^i=2$, $g_{b(i)-1}^i=1$, $r_{b(i)-1}^i>r_{b(i)+1}^i$, and in addition $b(i)-2=1$ and $r_1^i>r_{b(i)}^i$, let $$PF^{i+1}=\operatorname{dinvchange}(r_{d(i)}^i,r_{b(i)-1}^i,PF^i).$$
\begin{center}\begin{tabular}{rcl}\multirow{2}{*}{\setlength{\unitlength}{.13in}\begin{picture}(14,4)\linethickness{2pt}\multiput(0, 0)(2, 0){8}{\line(0, 1){4}}\multiput(0, 0)(0, 2){3}{\line(1, 0){14}}{\put(1,3){\makebox(0,0){$\mathbf{{r}}$}}}{\put(3,3){\makebox(0,0){$\mathbf{{r\prime}}$}}}{\put(5,3){\makebox(0,0){$\mathbf{{b}}$}}}{\put(7,3){\makebox(0,0){$\mathbf{{r\prime\prime}}$}}}{\put(9,3){\makebox(0,0){$\mathbf{{\dots}}$}}}{\put(11,3){\makebox(0,0){$\mathbf{{d}}$}}}{\put(13,3){\makebox(0,0){$\mathbf{{\dots}}$}}}{\put(1,1){\makebox(0,0){$\mathbf{{0}}$}}}{\put(3,1){\makebox(0,0){$\mathbf{{1}}$}}}{\put(5,1){\makebox(0,0){$\mathbf{{1}}$}}}{\put(7,1){\makebox(0,0){$\mathbf{{2}}$}}}{\put(9,1){\makebox(0,0){$\mathbf{{1+}}$}}}{\put(11,1){\makebox(0,0){$\mathbf{{0}}$}}}{\put(13,1){\makebox(0,0){$\mathbf{{1+}}$}}}\linethickness{1pt}
\curve(3,4,7.5,5,12,4)\thicklines\put(11.1,4.4){\vector(3,-1){.8}}\linethickness{1pt}
\curve(0,4,5.5,5.2,11,4)\thicklines\put(0.6,4.5){\vector(-3,-2){.8}}
\end{picture}}&&\multirow{2}{*}{\setlength{\unitlength}{.13in}\begin{picture}(14,4)\linethickness{2pt}\multiput(0, 0)(2, 0){8}{\line(0, 1){4}}\multiput(0, 0)(0, 2){3}{\line(1, 0){14}}{\put(1,3){\makebox(0,0){$\mathbf{{d}}$}}}{\put(3,3){\makebox(0,0){$\mathbf{{b}}$}}}{\put(5,3){\makebox(0,0){$\mathbf{{r\prime\prime}}$}}}{\put(7,3){\makebox(0,0){$\mathbf{{\dots}}$}}}{\put(9,3){\makebox(0,0){$\mathbf{{r}}$}}}{\put(11,3){\makebox(0,0){$\mathbf{{r\prime}}$}}}{\put(13,3){\makebox(0,0){$\mathbf{{\dots}}$}}}{\put(1,1){\makebox(0,0){$\mathbf{{0}}$}}}{\put(3,1){\makebox(0,0){$\mathbf{{1}}$}}}{\put(5,1){\makebox(0,0){$\mathbf{{2}}$}}}{\put(7,1){\makebox(0,0){$\mathbf{{2+}}$}}}{\put(9,1){\makebox(0,0){$\mathbf{{0}}$}}}{\put(11,1){\makebox(0,0){$\mathbf{{1}}$}}}{\put(13,1){\makebox(0,0){$\mathbf{{1+}}$}}}\end{picture}}\\&$\rightarrow $&\end{tabular}
\end{center}\vspace{1cm}
\end{enumerate}
\end{procedure}
\begin{lemma} The previous parking functions are well defined.
\end{lemma}
\begin{proof} Repeatedly, we use that $r_{d(i)}^i$ is in the dinv set of $r_{b(i)}^i$, so we may decrease the dinv of $r_{b(i)}^i$. In addition: 
\begin{enumerate}
\item Note that $r_{d(i)}^i$ is contained in the dinv set of $r_2^i$ when $r_2^i>r_{d(i)}^i$ and is contained in the dinv set of $r_1^i$ when $r_2^i<r_{d(i)}^i$ so above dinv decreases are possible.
\item If $g_{b(i)+1}^i=1$, then by definition of $b(i)$, $b(i)+1=a(i)$, so $a(i)$ exists. Certainally, $r_{b(i)}^i$ is in the degree set of $r_{a(i)}^i$, but not its dinv set. Moreover, since we assume $b(i)\neq 2$ (or else we would apply the first step), either $r_2^i>r_{a(i)}$ and thus is in the degree set but not the dinv set of $r_{a(i)}^i$ or $r_1^i<r_2^i<r_{a(i)}$ and thus $r_1^i$ is in the degree set. Either way, we may increase the dinv of $r_{a(i)}^i$.
\item Say $r_j^i$ is the last element in first diagonal occurring before $r_{b(i)+1}^i$ in $PF^i$. Then similarly to the previous construction, either $r_j^i<r_{b(i)+1}^i$ is in the degree set of $r_{b(i)+1}^i$ or $r_{j+1}^i>r_j^i>r_{b(i)+1}^i$ is in the degree set of $r_{b(i)+1}^i$ in addition to $r_{b(i)}^i$.
\item Notice that $r_{b(i)-1}^i$ and $r_{b(i)}^i$ are in the degree set of $r_{b(i)+1}^i$.
\item Since $b(i)-2\neq 1$, $r_2^i$ is distinct from $r_{b(i)-1}^i$. Then $r_1^i$ or $r_2^i$ as well as $r_{b(i)-1}^i$ are in the degree set of $r_{b(i)}^i$. Since $r_{b(i)-1}^i>r_{b(i)+1}^i>r_{b(i)}^i>r_{d(i)}^i$, $r_{d(i)}^i$ is in the dinv set of $r_{b(i)-1}^i$.
\item Notice that $r_{b(i)-1}^i$ and $r_1^i$ are in the degree set of $r_{b(i)}^i$. Again $r_{d(i)}^i$ is in the dinv set of $r_{b(i)-1}^i$.
\item Notice that $r_{1}^i>r_{b(i)}^i>r_{d(i)}^i$ and $0$ are in the degree set of $r_{d(i)}^i$. Again $r_{d(i)}^i$ is in the dinv set of $r_{b(i)-1}^i$.
\end{enumerate}
\end{proof}
Implicitly, we have assumed in Procedure \ref{k=1} that $g_{b(i)+1}^i\neq 0$. We must justify this assumption. Recall that $$T=\left\{\left[\begin{matrix}r_1&r_2&r_3&\dots&r_{n-2}&r_{n-1}&r_n\\g_1&g_2&g_3&\dots&g_{n-2}& 1&0\end{matrix}\right]\in \mathcal{C}^{\tau}_{n-1,1}: r_{n-1}>r_n\right\}$$ If $PF\in \mathcal{C}_{n-1,1}^{\tau}$ and $g_{b(i)+1}^i=0$, then $b(i)=n-1$. Since $r_{b(i)}^i>r_{d(i)}^i=r_n$, if $g_{b(i)+1}^i=0$, $PF\in T$. Thus we may assume that when $i=1$, Procedure \ref{k=1} is applied on a parking function where $g_{b(i)+1}^i\neq 0$. 
\begin{lemma} Every iteration of Procedure \ref{k=1} (except perhaps the last) results in a parking function satisfying the recursive condition.
\end{lemma}
\begin{proof}
For each case, we identify an element after $b(i+1)$ and before $d(i+1)$ in $PF^{i+1}$. 
\begin{enumerate}
\item[(1)\phantom{-(6)}] In the first case, we construct the final parking function and do not need this assumption. Thus in all future cases we assume $b'(i)$ exists. (If $b'(i)$ does not exist, then $PF^{i+2}$ will be the image of $PF^{i+1}$ after applying (1). As we just observed, (1) does not require the assumption.)
\item[(2)\phantom{-(6)}] $r_{b'(i)}^i$ and $r_{b(i)}^i$ can't both be to the right of $r_{a(i)}^i$ in $PF^{i+1}$ (or we would have increased the dinv of $r_{a(i)}^i$ by more than one.) Thus $r_{b(i+1)}^{i+1}=r_{b'(i)}^i$ is to the left of $r_{a(i)}^i$ in $PF^{i+1}$.
\item[(3)-(4)] $r_{b(i)+1}^i$ is to the right of $r_{b(i+1)}^{i+1}$ and the left of $r_{d(i+1)}^{i+1}$.
\item[(5)-(7)] $r_{b(i)}^i$ and $r_{b(i)+1}^{i}$ are also $r_{b(i+1)}^{i+1}$ and $r_{b(i+1)+1}^{i+1}$ respectively.
\end{enumerate}
\end{proof}
We need to show that our procedure in fact terminates. We do this by showing that each step reduces the number of big elements in the first diagonal and the first part. Thus eventually we are in the first case and construct a final parking function $f^3(PF^1)$.
\begin{lemma} Every iteration of Procedure \ref{k=1} (except perhaps the last) results in a parking function which contains more big elements in the second part and the first diagonal than the previous one.
\end{lemma}
\begin{proof}
Notice that all of our cases except the last, the elements from the last part of $PF^i$ are in the last part of $PF^{i+1}$. (We never change the dinv of these elements, and thus we can show recursively that they remain to the right of $r^i_{d(i)}$ in $PF^{i+1}$.) Thus in all but the last case, we need only identify a single big element in the first part and diagonal of $PF^i$ which occurs in the second part of $PF^{i+1}$. For cases (2)-(4), $r_{b(i)}^i$ is this element. For cases (5) and (6) this is $r_{b(i)-1}^i$.
For (7), we exchange the position of $r_{1}^i$ and $r_{d(i)}^i$ and add the element $r_{2}^i$ on top of $r_{1}^i$ in the top part of $PF^i$ to create $PF^{i+1}$. Thus every element in the top part of $PF^i$ besides $r_{d(i)}$ remains in the top part of $PF^{i+1}$ (and in fact to the right of $r_2^i$).
\end{proof}
\begin{lemma} The image of $f_3$ is a subset of $\mathcal{C}^\tau_{1,n-1}$.
\end{lemma}
\begin{proof} We assume $PF^i$ has no big elements in the first part and first diagonal besides possibly a big element in its first column. We must consider the composition when we then apply Step (1) of Procedure \ref{k=1}. 
If there is no big element in the first column and first diagonal, then $r_2^i<r_{d(i)}^i$ and thus we consider $\operatorname{dinvdec}(r_1^i,PF^i)$. Since $r_1^i$ must be less than $r_{d(i)}^i$, we have that $r_1^i$ will fall after $r_{d(i)}^i$ in $PF^i$. Moreover, any element in the first diagonal of $PF^i$ is either already in the second part of $PF^i$ or is smaller than $r_{d(i)}^i$ and thus will be to the right of $r_1^i$ in $PF^{i+1}$. Thus the result is a parking function with $r_1^i$ \textit{directly} to the right of $r_{d(i)}^i$ in $PF^{i+1}$ and thus $f_3(PF)$ has composition $[1,n-1]$.
If there is a big element in the first column and first diagonal, in particular, this element is $r_2^i$. Thus $f_3(PF)=\operatorname{dinvdec}(r_2^i,PF^i)$. Decreasing the dinv of $r_2^i$ in $PF^i$ results in a parking function with $r_2^i$ on top of $r_{d(i)}^i$, since the remaining elements between $r_2^i$ and $r_{d(i)}^i$ in $PF^i$ are either in at least the second diagonal or are in the first diagonal and smaller than $r_{d(i)}^i$ and thus than $r_2^i$. Thus the first three elements (in order) in $f_3(PF)$ are $r_1^i$, $r_{d(i)}^i$, and $r_2^i$ and thus $f_3(PF)$ has composition $[1,n-1]$.
\end{proof}
\begin{lemma}$f_3$ respects ides.
\end{lemma}
\begin{proof} By Lemma \ref{lem:easyprop} we need only ensure that any two elements $r$ and $r+1$ in the same diagonal do not change order when we apply Procedure \ref{k=1}. We use repeatedly that when we decrease the dinv of $r_{b(i)}^i$, any element in the first diagonal of $PF^i$ between $r_{b(i)}^i$ and $r_{d(i)}^i$, call it $r\prime$, is smaller than $r_{d(i)}^i$ by definition of $b(i)$. This gives us that $r_{b(i)}^i\neq r\prime \pm 1$ and thus our decrease in dinv does not change the ides. 
\begin{enumerate}
\item[(1)\phantom{-(6)}] The proof in this case is almost identical to the proof of Lemma \ref{lem:f1wt}, except that we replace $r_3^i$ with $r_{d(i)}^i$.
\item[(2)\phantom{-(6)}] Increasing the dinv of $r_{a(i)}^i$ moves $r_{a(i)}^i$ just past $r_{b(i)}^i$. If $r_{a(i)}^i-1$ is in the first diagonal and just to the left of $r_{b(i)}^i$ in $PF^i$, both $r_{a(i)}^i$ and $r_{b(i)}^i$ are in its dinv set. Although $r_{a(i)}^i$ and $r_{b(i)}^i$ change order in $PF^{i+1}$ they both remain in its dinv set and thus both remain to its right. If $r_{a(i)}^i+1$ is in the first diagonal and to the left of $r_{b(i)}^i$, $r_{a(i)}^i+1$ is in the degree set of $r_{a(i)}^i$ but not its dinv set. The dinv set of $r_{a(i)}^i$ increases by exactly one element, $r_{b(i)}^i$, as we form $PF^{i+1}$, so $r_{a(i)}^i+1$ remains to the left of $r_{a(i)}^i$ in $PF^{i+1}$. Our previous argument gives that the following dinv decrease of $r_{b(i)}^i$ will also not change the ides.
\item[(3)-(4)] A similar argument to the one given for (2) gives that $r_{b(i)+1}^i$ will move just to the right of $r_{b(i)-1}^i$ without changing the word of the parking function.
\item[(5)-(6)] Like when we decrease the dinv of $r_{b(i)}^i$, when we decrease the dinv of $r_{b(i)-1}^i$ we do not switch the relative order of $r_{b(i)-1}^i$ and any elements within 1 of $r_{b(i)-1}^i$ that may sit after $r_{b(i)}^i$ and before $r_{d(i)}^i$ in $PF^i$, since $r_{b(i)-1}^i>r_{b(i)+1}^i>r_{b(i)}^i>r_{d(i)}^i$ so $r_{b(i)-1}^i$ is a big element and thus not numerically adjacent to any element occuring in the first row between itself and $r_{d(i)}^i$. While it will reverse order with $r_{b(i)}^i$, the previous inequalities also give us that the difference between $r_{b(i)-1}$ and $r_{b(i)}$ is more than one and thus this shift does not change the ides.
\item[(7)\phantom{-(6)}] Here we reverse the relative order of $r_2^i$ and $r_{b(i)}^i$ as well as the order of $r_1^i$ and $r_{d(i)}^i$. Since $r_{d(i)}^i<r_{b(i)}^i<r_{1}^i<r_2^i$, neither of these two pairs are numerically adjacent and thus there is no change in ides. 
\end{enumerate}
\end{proof}
\begin{lemma} $f_3(\mathcal{A}^\tau_{n-1,1}\backslash T)\cap S=\emptyset$
\end{lemma}
\begin{proof} Recall that $S=f_1(\mathcal{A}_{2,n-2}^\tau)$, where $$f_1(PF)=\begin{cases}\operatorname{dinvdec}(r_2,PF)& r_2>r_3\\
\operatorname{dinvdec}(r_1,PF)& r_2<r_3\end{cases}.
$$ Using notation defined after this map, notice that $r_3$ is exactly $r_{d(1)}^1$ and there are no elements between $r_1^1$ and $r_{d(1)}^1$ except $r_2^i$, which is in the first column.
Moreover, in this notation $$f_1(PF^1)=\begin{cases}\operatorname{dinvdec}(r_2^1,PF^1)& r_2^1>r_{d(1)}^1\\
\operatorname{dinvdec}(r_1^1,PF^1)& r_2^1<r_{d(1)}^1\end{cases},
$$ exactly the final map that gives us $f_3(PF)$ for parking functions not in $T$. Thus, we may define $$h(PF^1)=\begin{cases}\operatorname{dinvinc}(r_2^1,PF^1)& r_2^1>r_{1}^1\\
\operatorname{dinvinc}(r_1^1,PF^1)& r_2^1<r_{1}^1\end{cases},
$$ and notice that it is not only the inverse of $f_1$ on $S$, but the inverse of our final step in Procedure \ref{k=1} on $\mathcal{A}^{\tau}_{1,n-1}\backslash S$. Thus to prove that the image $f_3$ is disjoint from $S$, we need to prove that the elements that we construct with Procedure \ref{k=1} before applying Step (1) of Procedure \ref{k=1} are not elements with composition $[2,n-2]$. 
Assume we have an element $PF^i$ with no large elements in the first part and first diagonal besides a possible element in the first column. Clearly, if $i=1$, then $PF^1$ is an element with composition $[n-1,1]$. If we assume that $i\neq 1$, assume for the sake of contradiction that $PF^i$ has composition $[2,n-2]$. Then consider $PF^{i-1}$.
In all but the second case, the resulting parking function has at least two elements in the first part distinct from $r_1^{i-1}$, each marked in our domino diagrams. 
In the second case, assume that $PF^{i}=\operatorname{dinvchange}(r_{a({i-1})}^{i-1},r_{b({i-1})}^{i-1},PF^{i-1}).$
Then the two elements in the first part of $PF^i$ must be $r_1^i$ and $r_2^i=r_{a(i-1)}^{i-1}$. Thus $PF^{i-1}$ begins with the cars (given in order) $r_1^i$, $r_{d(i)+1}^i$, $r_2^i$, and $r_{d(i)}^i$. But then the only big element in the first part and diagonal in $PF^{i-1}$ is $r_{d(i)+1}^i$, which is in the first column of $PF^{i-1}$. This implies that we should have applied the first step of Procedure \ref{k=1} to $PF^{i-1}$, a contradiction.
\end{proof}
\begin{lemma} $f$ is a bijection.
\end{lemma}
\begin{proof}
Using $f$, we have an explicit bijection between the left and right hand side. Since we have already given the inverse map for $f_1$ and $f_2$ and how to determine whether a given parking function is the image of $f_1$, $f_2$, or $f_3$, it remains to show that $f_3$ is bijective. To that end, we give a procedure that inverts $f_3$.
\begin{procedure} Let $PF=PF^1$ be an element in $\mathcal{A}^{\tau}_{1,n-1}\backslash S$. Begin with $PF^2=h(PF^1)$. As discussed above, the result is a parking function with composition different from $[2,n-2]$. Afterwards, repeat the following until the resulting parking function has composition $[n-1,1]$. (Our numbering below corresponds to the numbering in Procedure \ref{k=1} such that these cases give the inverse of the step with the same number in Procedure \ref{k=1}.)
\begin{enumerate}
\item[(2)\phantom{-(6)}] If $r_{p(i)}$ does not exist and $m(i)\neq 2$ and $r_{m(i)}^i<r_{d(i)}^i$, then let $$PF^{i+1}=\operatorname{dinvchange}(r_{d(i)+1}^i,r_{m(i)}^i,PF^i).$$
\item[(3)-(4)] If $r_{p(i)}^i$ exists, let $$PF^{i+1}=\operatorname{dinvchange}(r_{d(i)+1}^i,r_{p(i)}^i,PF^i).$$
\item[(5)-(6)] If $r_{p(i)}^i$ does not exist, but $r_{m(i)}^i>r_{d(i)}^i$, let $$PF^{i+1}=\operatorname{dinvchange}(r_{d(i)+1}^i,r_{m(i)}^i,PF^i).$$
\item[(7)\phantom{-(6)}] If $r_{p(i)}^i$ does not exist, $r_{m(i)}^i<r_{d(i)}^i$, and $m(i)=2$ , then let $$PF^{i+1}=\operatorname{dinvchange}(r_{d(i)+1}^i,r_{1}^i,PF^i).$$
\end{enumerate}
\end{procedure}
\end{proof}
\section{Final Results and Applications}
Thus far, we define $f$ on a subset of the parking functions with composition $[n-1,1]$ and $[2,n-2]$. As observed earlier, because $f$ respects the diagonal word, we may expand the domain of $f$ to include parking functions with compositions $$[c_1,\cdots,c_i,1,c_{i+2},\cdots,c_k]\text{ or }[c_1,\cdots,c_{i-1},2,c_i-1,c_{i+2},\cdots,c_k]$$ by insisting that $f$ only changes the relative order of elements in their $i^{\text{th}}$ or ${i+1}^{\text{th}}$ parts. Using this expanded definition of $f$ we have the following final result.
\begin{theorem}\label{thm:bigone} Let $$c=[c_1,\cdots,c_i,1,c_{i+2},\cdots,c_k]=[c',c_i,1,c''],$$ where we use $c'$ ($c'')$ for the sequence giving the first $i-1$ (respectively last $k-i-1$) parts of $c$.
$f$ has the following properties:
\begin{itemize}
\item $f$ is a bijection from $\mathcal{A}_{[c',c_i,1,c'']}\cup \mathcal{A}_{[c',2,c_{i}-1,c'']}$ to $\mathcal{A}_{[c',1,c_i,c'']}\cup \mathcal{A}_{[c',c_{i}-1,2,c'']}$
\item $f$ decreases the dinv by exactly one.
\item $f$ respects ides, area, and diagonal word.
\end{itemize}
\end{theorem}
\begin{corollary}$$\mathcal{C}^\tau_{n-1, 1}+\mathcal{C}^\tau_{2, n-2}=\frac{1}{q}(\mathcal{C}^\tau_{n-2, 2}+\mathcal{C}^\tau_{1, n-1})$$
\end{corollary}
In summary, this paper demonstrates that the commutativity relations for the modified Hall Littlewood operator $C_{[i,j]}$ operator are the same as those for the parking functions with composition $[i,j]$ when $i$ or $j$ is $1$. Future research is required to give the same commutativity relations when neither $i$ nor $j$ is 1. Experimental results suggest that such bijections should also respect the diagonal word; thus it is reasonable to believe that such proofs would give the full community relations (i.e. including for partitions of any number of parts).
The current result is enough to simplify the proof of several recently proved theorems on parking functions whose word is a particular type of shuffle. 
\begin{definition}
We say that a parking function $PF$ is a \textbf{shuffle} of $[a_1,\cdots, a_k]$ and $[b_1, \cdots, b_{n-k}]$ (or is in $S([a_1,\cdots, a_k], [b_1, \cdots, b_{n-k}])$) if for all $i$, $a_i$ occurs before $a_{i+1}$ and $b_i$ occurs before $b_{i+1}$ in the word of $PF$. 
\end{definition}
Recently the following two theorems were proved in \cite{EHshuffle} and \cite{GXZ} using $\langle,\rangle$ to denote the Hall scalar product:
\begin{theorem}
$$\langle\nabla C_c 1,e_i h_{n-i}\rangle=\sum_{\substack{PF\in S([1,\cdots, n-i], [n, \cdots, n-i+1])\\\operatorname{comp}(PF)=c}} t^{\text{area}(PF)} q^{\text{dinv}(PF)} $$
\end{theorem}
\begin{theorem}
$$\langle\nabla C_c 1,h_i h_{n-i}\rangle=\sum_{\substack{PF\in S([1,\cdots, i], [i+1, \cdots, n])\\\operatorname{comp}(PF)=c}} t^{\text{area}(PF)} q^{\text{dinv}(PF)} $$
\end{theorem}
Proof of both theorems demonstrated the existence of recursions satisfied by both sides of the equations along with the equality of a small family of base cases. In both instances, the relevant recursion was specifically different when $c_1=1$ (where $c_1$ is the first part of $c$) and required a separate combinatorial proof from when $c_1>1$. For the latter theorem, the proof when $c_1=1$ is significantly harder than when $c_1>1$. This paper allows us to eliminate consideration of partitions with first part size $1$, except when $p=\{1,\cdots,1\}$, which was already a trivial base case in the previous proofs.
\section{Funding}
This work was supported by the National Science Foundation.
\bibliographystyle{plain}
\bibliography{k1bib}

\begin{thebibliography}{1}

\bibitem{nabla}
F.~Bergeron and A.~M. Garsia.
\newblock Science fiction and {M}acdonald's polynomials.
\newblock In {\em Algebraic methods and {$q$}-special functions ({M}ontr\'eal,
  {QC}, 1996)}, volume~22 of {\em CRM Proc. Lecture Notes}, pages 1--52. Amer.
  Math. Soc., Providence, RI, 1999.

\bibitem{Remarkable}
A.~M. Garsia and M.~Haiman.
\newblock A remarkable {$q,t$}-{C}atalan sequence and {$q$}-{L}agrange
  inversion.
\newblock {\em J. Algebraic Combin.}, 5(3):191--244, 1996.

\bibitem{Cop}
A.~M. Garsia and C.~Procesi.
\newblock On certain graded {$S\sb n$}-modules and the {$q$}-{K}ostka
  polynomials.
\newblock {\em Adv. Math.}, 94(1):82--138, 1992.

\bibitem{GXZ}
A.~M Garsia, G.~Xin, and M.~Zabrocki.
\newblock Hall littlewood operators in the theory of parking functions and
  diagonal harmonics.
\newblock {\em Int. Math. Res. Not.}, (6):1264--1299, 2012.

\bibitem{HHLRU}
J.~Haglund, M.~Haiman, N.~Loehr, J.~B. Remmel, and A.~Ulyanov.
\newblock A combinatorial formula for the character of the diagonal
  coinvariants.
\newblock {\em Duke Math. J.}, 126(2):195--232, 2005.

\bibitem{HMZConj}
J.~Haglund, J.~Morse, and M.~Zabrocki.
\newblock A {C}ompositional {S}huffle {C}onjecture {S}pecifying {T}ouch
  {P}oints of the {D}yck {P}ath.
\newblock {\em Canad. J. Math.}, 64(4):822--844, 2012.
\newblock \url{http://dx.doi.org/10.4153/CJM-2011-078-4}.

\bibitem{Haiman}
Mark Haiman.
\newblock Hilbert schemes, polygraphs and the {M}acdonald positivity
  conjecture.
\newblock {\em J. Amer. Math. Soc.}, 14(4):941--1006 (electronic), 2001.

\bibitem{EHshuffle}
A.~Hicks.
\newblock Two parking function bijections:a sharpening of the $q,t$-{C}atalan
  and {S}hr\"oder theorems.
\newblock {\em Int. Math. Res. Not.}, (13):3064--3088, 2012.

\end{thebibliography}

\end{document}